\documentclass{amsart}
\usepackage[a4paper,twoside,inner=2cm,outer=2cm]{geometry}
\usepackage{amssymb,amsmath,amscd}
\usepackage[all,cmtip]{xy}   
\usepackage{tikz}
\usetikzlibrary{matrix,arrows}
\usepackage{graphicx}  
\usepackage{float}  
\usepackage{multirow}  
\usepackage{booktabs}  
\usepackage{subcaption}  
\usepackage{bookmark}  

\theoremstyle{definition}
 \newtheorem{dfn}{Definition}[section]
 
\theoremstyle{remark}
 \newtheorem{rmk}{Remark}[section]
 \numberwithin{equation}{section}
\theoremstyle{plain}
\newtheorem{thm}{Theorem}[section]
\newtheorem{prop}{Proposition}[section]
\newtheorem{lem}{Lemma}[section]
\newtheorem{cor}{Corollary}[section]
\newtheorem{fact}{Fact}[section]

\renewcommand{\leq}{\leqslant}
\renewcommand{\geq}{\geqslant}
\renewcommand{\setminus}{\smallsetminus}

\title{Classification and equivariant cohomology of circle actions on 3d manifolds}



\author[He]{\bfseries Chen He}

\address{
Department of Mathematics \\ 
Northeastern University   \\ 
Boston\\
USA}
\email{he.chen@husky.neu.edu}

\begin{document}

\vspace{18mm} \setcounter{page}{1} \thispagestyle{empty}

\begin{abstract}
The classification of Seifert manifolds was given in terms of numeric data by Seifert \cite{Se33}, and then generalized by Orlik and Raymond \cite{Ra68,OR68} to circle actions on closed 3d manifolds. In this paper, we further generalize the classification to circle actions on 3d manifolds with boundaries by adding a numeric parameter and a union of cycle graphs. Then we describe the equivariant cohomology of 3d manifolds with circle actions in terms of ring, module and vector-space structures. We also compute equivariant Betti numbers and Poincar\'e series for these manifolds and discuss the equivariant formality. 
\end{abstract}

\maketitle

\tableofcontents

\section{Introduction}
\vskip 15pt
The classification of closed 3d manifolds with ``nice" decompositions into circles was given by Seifert \cite{Se33} in terms of principal Euler number $b$, orientability $\epsilon$ and genus $g$ of the underlying 2d orbifolds, and pairs of coprime integers $(m_i,n_i)$ called Seifert invariants. Hence these manifolds were given the name Seifert manifolds. 

Later, the classification was generalized by Orlik and Raymond \cite{Ra68, OR68} to circle actions on closed 3d manifolds allowing fixed points and special exceptional orbits. Orlik and Raymond found that in their case the underlying 2d orbifolds have circle boundaries contributed by the fixed points and special exceptional orbits. Hence, besides the four types of numeric data used by Seifert, two more types of numeric data were introduced by Orlik and Raymond: the number $f$ of fixed components and the number $s$ of special exceptional components. Then Orlik and Raymond proved:

\newtheorem*{OR}{\textup{\textbf{Theorem}}}
\begin{OR}[Orlik-Raymond classification of closed 3d  $S^1$-manifolds, \cite{Ra68,OR68}]
Let $S^1$ act effectively and smoothly on a closed, connected smooth 3d manifold $M$. Then the orbit invariants
	\[
	\big\{b;(\epsilon,g,f,s);(m_1,\,n_1),\ldots,(m_r,\,n_r)\big\}
	\]
determine $M$ up to equivariant diffeomorphisms, subject to certain conditions. Conversely, any such set of invariants can be realized as a closed 3d manifold with an effective $S^1$-action.
\end{OR}

The first goal of this paper is to further generalize the Orlik-Raymond Classification Theorem to circle actions on compact 3d manifolds, allowing boundaries. By the classification of circle actions on closed 2d manifolds, those boundaries have to be tori $\mathbb{T}$, spheres $S^2$, projective planes $\mathbb{R}P^2$ or Klein bottles $K$. Our approach starts with a careful discussion on the equivariant neighbourhoods of non-principal orbits and boundaries. We find that the underlying orbit spaces are 2d orbifolds with boundaries, and possibly with corners. In order to generalize the Orlik-Raymond Classification Theorem to 3d circle-manifolds with boundaries, we will cap off the boundaries by standard fillings and then pass to the Orlik-Raymond case of no boundary. As a result, let $t$ be the number of torus boundaries and $\mathcal{G}$ be a union of labelled cycle graphs to keep track of the boundary types $S^2,\,\mathbb{R}P^2,\,K$, we get:

\newtheorem*{classifyB}{\textup{\textbf{Theorem \ref{classifyB}}}}
\begin{classifyB}
Let the circle group $S^1$ act effectively and smoothly on a compact, connected 3d manifold $M$, possibly with boundary. Then the orbit invariants
\[
\big\{b;(\epsilon,g,f,s,t);(m_1,\,n_1),\ldots,(m_r,\,n_r);\mathcal{G}\big\}
\]
consisting of numeric data and a collection of labelled cycle graphs, 
determine $M$ up to equivariant diffeomorphisms, subject to certain conditions. Conversely, any such set of invariants can be realized as a 3d manifold with an effective $S^1$-action.
\end{classifyB}

Using the Orlik-Raymond Theorem, one can compute the fundamental groups, ordinary homology and cohomology with $\mathbb{Z}$ or $\mathbb{Z}_p$ coefficients for closed 3d $S^1$-manifolds, (cf.\cite{JN83, BHZZ00, BLPZ03, BZ03}). Using the generalized classification Theorem \ref{classifyB}, one can also compute those non-equivariant topological invariants to 3d $S^1$-manifolds with boundaries. But in this paper, we are more interested in equivariant topological invariants.

So the second goal of this paper is to describe the $\mathbb{Q}$-coefficient equivariant cohomology of any compact 3d manifold $M$ with circle action. Our main strategy is to apply the equivariant Mayer-Vietoris sequence to a decomposition of the manifold $M$ into a fixed-point-free part and a neighbourhood of the fixed-point set. Then we get 

\newtheorem*{equivCohom1}{\textup{\textbf{Theorem \ref{equivCohom1}}}}
\begin{equivCohom1}
Let $M$ be a compact connected 3d manifold(possibly with boundary) with an effective $S^1$-action, and $F$ be its fixed-point set(possibly empty), then there is a short exact sequence of cohomology groups in $\mathbb{Q}$ coefficients: 
\[
0\rightarrow H^*_{S^1}(M) \rightarrow H^*(M/{S^1}) \oplus 
\Big( \mathbb{Q}[u]\otimes H^*(F) \Big) \rightarrow   
H^*(F) \rightarrow 0
\]
\end{equivCohom1}
Using this theorem, we can describe the ring, module and vector-space structures of the equivariant cohomology $H^*_{S^1}(M)$ in details. Furthermore, we will calculate equivariant Betti numbers and Poincar\'{e} series, and discuss a numeric condition for equivariant formality.

The following is a brief summary of each section:

In Section 2, we recall the folklore classification of circle actions on 2d manifolds and the Orlik-Raymond classification of circle actions on closed 3d manifolds in terms of numeric data.

In Section 3, we generalize the classification theorem to circle actions on 3d manifolds possibly with boundaries.

In Section 4, we describe the equivariant cohomology of 3d manifolds with circle actions, then calculate the equivariant Betti numbers and Poincar\'{e} series, and discuss equivariant formality.

\vskip 20pt
\section{$S^1$-actions on 2d manifolds and closed 3d manifolds}
\vskip 15pt

In this section, we will recall the classification of effective $S^1$-actions on manifolds in dimension 2 and 3. All these results are well known, and can be found in greater details from the original papers by Orlik and Raymond \cite{Ra68,OR68} or the notes and books \cite{Or72,JN83,Au04,Ni05}.

\subsection{Some basic facts about group actions on manifolds}
Throughout the paper, we always assume that a manifold $M$ is compact, smooth and connected, and a group $G$ is compact, unless otherwise mentioned. For convenience, we will denote a $G$-action on $M$ as $G \curvearrowright M$. The quotient $M/G$ is called the \textbf{orbit space} of the $G$-action on $M$. For any point $x$ in $M$, let $G_x=\{g\in G\mid g\cdot x=x\}$ be its stabilizer. We denote $M^G=\{x\in M\mid G_x=G\}$ for the set of fixed points. If $G_x=G$ for every $x\in M$, we say that the $G$-action on $M$ is \textbf{trivial}. If $G_x=\{1\}$ for every $x\in M$, we say that the $G$-action on $M$ is \textbf{free}. If the intersection $\cap_{x\in M}G_x=\{1\}$, we say that the $G$-action on $M$ is \textbf{effective}. Throughout this paper, group actions are usually assumed to be effective, unless otherwise mentioned.

For any orbit $G\cdot x$, let $V_x$ be an orthogonal complement of $T_x(G\cdot x)$ in $T_x M$. The infinitesimal action of $G_x$ on $T_x M$ gives a linear \textbf{isotropy representation} $G_x \curvearrowright V_x$. Then the normal bundle of the orbit $G\cdot x$ can be written as 
\[
G\times_{G_x}V_x=\big\{[g,v] \mid (g,v) \sim (gh,h^{-1}v) \mbox{ for any } h\in G\big\}
\] 
with a canonical $G$-action induced from the $G$-principal bundle $G\times V_x$.

The following theorem, proved by Koszul \cite{Ko53}, equivariantly identifies the normal bundle with the tubular neighbourhood of an orbit $G\cdot x$.

\begin{thm}[The slice theorem, \cite{Ko53}]
There exists an equivariant exponential map
\[
 \exp: G\times_{G_x}V \longrightarrow M
\]
which is an equivariant diffeomorphism from an open neighbourhood of the zero section $G\times_{G_x} \{0\}$ in $G\times_{G_x}V_x$ to an equivariant neighbourhood of $G\cdot x$ in $M$.
\end{thm}

Thus, an equivariant neighbourhood of the orbit $G\cdot x$ can be specified in terms of the stabilizer $G_x$ and the isotropy representation of $G_x$ on the normal vector space. 

Similar to the ordinary non-equivariant case, the equivariant identification between normal bundles and neighbourhoods generalizes beyond single orbit to submanifold and boundary, cf. Kankaanrinta \cite{Ka07}.

\begin{thm}[Equivariant tubular neighbourhood,  \cite{Ka07}]
Let $N$ be a closed $G$-invariant submanifold of $M$, and $E$ be the normal $G$-vector bundle of $N$. There exists an equivariant exponential map
\[
 \exp: E \longrightarrow M
\]
which is an equivariant diffeomorphism from an open neighbourhood of the zero section in $E$ to an equivariant tubular neighbourhood of $N$ in $M$.
\end{thm}

\begin{thm}[Equivariant collaring neighbourhood, \cite{Ka07}]
Suppose a compact manifold $M$ has a $G$-action that extends compatibly to its boundary $\partial M$. There exists an equivariant exponential map
\[
 \exp: \partial M \times [0,\infty) \longrightarrow M
\]
which is an equivariant diffeomorphism from an open neighbourhood of the boundary $\partial M$ in $\partial M \times [0,\infty)$ to an equivariant collaring neighbourhood of $\partial M$ in $M$. 
\end{thm}

Since we only consider $S^1$-actions, there are three types of stabilizers, namely $\{1\},\, \mathbb{Z}/m, \,S^1$, whose resulting orbits will be called \textbf{principal}, \textbf{exceptional} and \textbf{singular} respectively. 
\begin{center}
\begin{tabular}{c c c c}
\toprule
 & Principal orbit & Exceptional orbit & Singular orbit\\
\midrule
Stabilizer $S^1_x$& $\{1\}$ & $\mathbb{Z}_m = \{e^{\frac{2\pi k i}{m} }, k = 1,2,\ldots,m\}$ & $S^1$\\
\midrule
Orbit $S^1\cdot x$ & $S^1$ & $S^1/\mathbb{Z}_m$ & $pt $\\
\bottomrule
\end{tabular}
\end{center}
Intuitively, exceptional orbits $S^1/\mathbb{Z}_m$ are shorter than regular orbits $S^1$. Singular orbits $S^1/S^1=pt$ are exactly the fixed points of the $S^1$-action.

Direct applications of the Slice Theorem, together with the compactness of $M$, leads to the following facts (cf. Audin \cite{Au04} Sec I.2):

\begin{fact}
If $S^1$ acts on a compact, connected manifold $M$, then
\begin{itemize}
\item
For any subgroup $H$ of $S^1$, the set $M_{(H)}=\{ x \in M \mid S^1_x = H\}$ of points with stabilizer $H$ is a submanifold of $M$. Moreover, $S^1/H$ acts freely on $M_{(H)}$. 
\item
There is a unique subgroup $H_0$ of $S^1$, such that the set $M_{(H_0)}$ is open and dense in $M$.
\item
The $S^1$-action on $M$ is effective if and only if the $H_0$ in the previous statement is the identity group $\{1\}$.
\item
If the $S^1$-action on $M$ is effective, then for every $x \in M$, the isotropy representation $S^1_x \curvearrowright V_x$ is also effective.
\end{itemize}
\end{fact}

Furthermore, based on the Theorem of equivariant tubular neighbourhood, the classification of effective $S^1$-manifolds at low dimensions can be done by listing all the possible equivariant neighbourhoods and the obstructions of patching them together to form a manifold. In dimension 1, there is only one compact effective $S^1$-manifold, the circle $S^1$ itself with the rotating action. In dimension 2 and 3, this approach is also successful, as we will recall in the next subsections.

\subsection{$S^1$-actions on 2d manifolds}
We begin by listing all the possible equivariant tubular neighbourhoods of orbits, which are the same as equivariant normal bundles according to the Slice Theorem. Then we try to patch these neighbourhoods together. The survey of this topic follows closely from Audin (\cite{Au04} Sec I.3).

Notice that in dimension 2, for an exceptional orbit $S^1/\mathbb{Z}_m$, its isotropic representation is of dimension 1. But there is only one such effective representation, namely the reflection $\mathbb{Z}_2 \overset{\text{reflect}}{\curvearrowright} \mathbb{R}$, which also forces the exceptional orbit to be $S^1/\mathbb{Z}_2$.

As for a singular orbit, i.e. a fixed point with stabilizer $S^1$, its isotropic representation is of dimension 2. The only effective $S^1$-representation of real dimension 2 is the rotation $S^1\overset{\text{rotate}}{\curvearrowright} \mathbb{C}$.

So we can summarize the list of all possible equivariant tubular neighbourhoods:

\begin{center}
\begin{tabular}{c c c c}
\toprule
 & Principal orbit & Exceptional orbit & Singular orbit\\
\midrule
Stabilizer $S^1_x$& $\{1\}$ & $\mathbb{Z}_2$ & $S^1$\\
\midrule
Orbit $S^1\cdot x$ & $S^1$ & $S^1/\mathbb{Z}_2$ & $pt $\\
\midrule
Isotropic repr & $\{1\}\curvearrowright \mathbb{R}$ &
$\mathbb{Z}_2 \overset{\text{reflect}}{\curvearrowright} \mathbb{R}$ &
$S^1 \overset{\text{rotate}}{\curvearrowright} \mathbb{C}$\\
\midrule
\multirow{3}{*} {Equiv nbhd} & $S^1 \times (-1,1) $ & $S^1 \times_{\mathbb{Z}_2} (-1,1)$ & $D=\{(x,y)\mid x^2+y^2<1\}$ \\
$U$ & \includegraphics[height = 1cm]{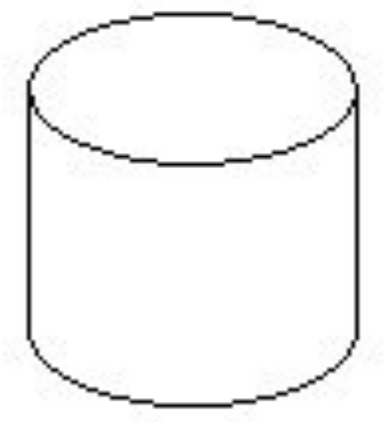}
& \includegraphics[height = 1cm]{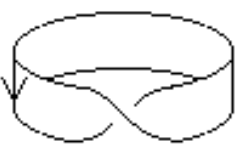} 
& \includegraphics[height = 1cm]{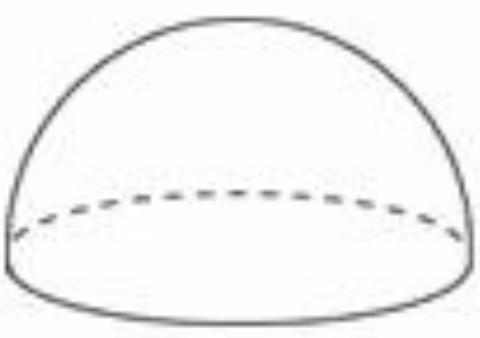}\\
& Cylinder & M\"{o}bius band & Disk\\
\midrule
Orbit space $U/S^1$ & $(-1,1)$ & $[0,1)$ & $[0,1)$\\
\bottomrule
\end{tabular}
\end{center}

To form a 2-dimensional closed manifold with effective $S^1$-action, we now just need to patch those equivariant pieces $S^1 \times (-1,1),\,S^1 \times_{\mathbb{Z}_2}(-1,1),\,D$ together by closing boundaries.

\begin{center}
\begin{tabular}{c c c c c c c c}
\includegraphics[height = 1cm]{cylinder.png} & + & \includegraphics[height = 1cm]{cylinder.png} & = & \includegraphics[height = 1cm]{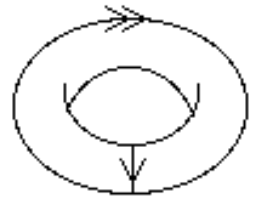} &  $\mathbb{T}^2$ & \textbf{Torus}\\
\includegraphics[height = 1cm]{hemisphere1.jpg} & + & \includegraphics[height = 1cm]{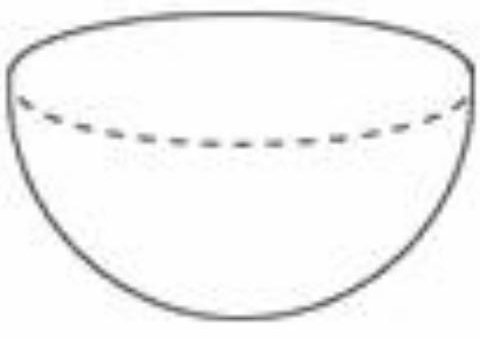}  & = & \includegraphics[height = 1cm]{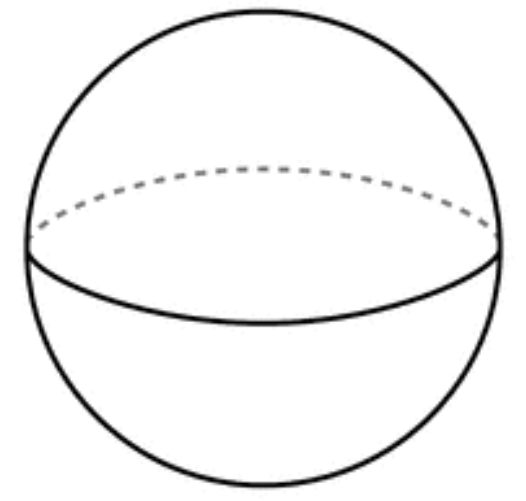} &  $S^2$ & \textbf{Sphere}\\
\includegraphics[height = 1cm]{hemisphere1.jpg} & + & \includegraphics[height = 1cm]{Mobius.png} & = & \includegraphics[height = 1cm]{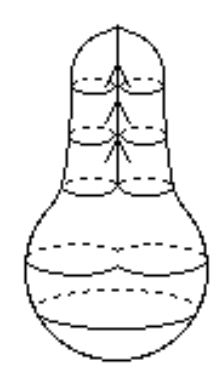} & $\mathbb{R}P^2$ & \quad \textbf{Projective plane}\\
\includegraphics[height = 1cm]{Mobius.png} & + & \includegraphics[height = 1cm]{Mobius.png} & = & \includegraphics[height = 1cm]{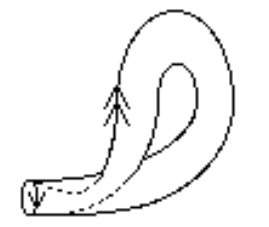} & $K$ & \textbf{Klein bottle}\\
\end{tabular}
\end{center}

In the above list of 2-dimensional closed manifolds with effective $S^1$-action, the projective plane $\mathbb{R}P^2$ and Klein bottle $K$ are non-orientable due to the existence of exceptional orbits $S^1/\mathbb{Z}_2$, but the torus $\mathbb{T}^2$ and the sphere $S^2$ are orientable.

Given a 2d compact connected effective $S^1$-manifold $M$, we can count its fixed points and exceptional orbits as $f$ and $s$ respectively. If we allow $M$ to have boundary, we can count the number of boundary components as $b$. Similarly, since the orbit space $M/S^1$ is a compact connected 1d manifold which is either a circle $S^1$ or an interval $I$, we can count the boundaries of $M/S^1$ as $\bar{b}$. Then we have the classification of the 2d compact connected effective $S^1$-manifolds:

\begin{thm}[Numeric classification of 2d $S^1$-manifolds]
Given a 2d compact connected effective $S^1$-manifold $M$, possibly with boundary, the integers $(b,f,s)$ determine $M$ up to $S^1$-diffeomorphism, and so do the integers $(\bar{b},f,s)$. 
\end{thm}
\begin{proof}
We have seen that there are three 2d effective $S^1$-manifolds with boundary: cylinder, disk and M\"{o}bius band, and four 2d effective $S^1$-manifolds without boundary: torus, sphere, projective plane and Klein bottle. The counting of boundary components as $b$ is straightforward.

To compute $(f,s)$, we first do this for cylinder, $f=0,\,s=0$; disk, $f=1,\,s=0$; M\"{o}bius band, $f=0,\,s=1$. For any one of the four closed 2d $S^1$-manifolds, we just add the $(f,s)$-vectors of its two patches. 

To understand the orbit spaces, we use the standard expressions for disk, $D$; cylinder, $S^1\times [-1,1]$; sphere, $S^2$; torus, $S^1\times S^1$. Their orbit spaces are $[0,1]$, $[-1,1]$, $[-1,1]$ and $S^1$ respectively.

For the orbit spaces of the rest types of the manifolds, notice that for a compact group $G$ and a compact subgroup $H$ that acts on a space $V$,  there is a relation between the $G$-orbit space and $H$-orbit space: $(G\times_H V)/G=V/H$. So the M\"{o}bius band, projective plane and Klein bottle, written respectively as $S^1\times_{\mathbb{Z}_2} [-1,1]$, $S^2/\mathbb{Z}_2$ and $S^1\times_{\mathbb{Z}_2} S^1$ will have $S^1$-orbit spaces $[0,1]$, $[0,1]$ and $[0,\pi]$ respectively.

Here is the complete list of the numeric data $(\bar{b},b,f,s)$:
\begin{center}
\begin{tabular}{c c c c c c c}
\toprule
Manifold  & Topological & Orbit space & $\# \partial(M/S^1)$ & $\#\partial M$ & $\# M^{S^1}$ & $\# M^{\mathbb{Z}_2}$\\
$M$ & expression & $M/S^1$ & $\bar{b}$ & $b$ & $f$ & $s$\\
\midrule
Disk & $D$ & $[0,1]$ & $2$ & $1$ & $1$ & $0$\\
\midrule
Cylinder & $S^1\times [-1,1]$ & $[-1,1]$ & $2$ & $2$ & $0$ & $0$\\
\midrule
M\"{o}bius band & $S^1\times_{\mathbb{Z}_2} [-1,1]$ & $[0,1]$ & $2$ & $1$ & $0$ & $1$\\
\midrule
Sphere & $S^2$ & $[-1,1]$ & $2$ & $0$ & $2$ & $0$\\
\midrule
Projective plane & $S^2/\mathbb{Z}_2$ & $[0,1]$ & $2$ & $0$ & $1$ & $1$\\
\midrule
Torus & $S^1\times S^1$ & $S^1$ & $0$ & $0$ & $0$ & $0$\\
\midrule
Klein bottle & $S^1\times_{\mathbb{Z}_2} S^1$ & $[0,\pi]$ & $2$ & $0$ & $0$ & $2$\\
\bottomrule
\end{tabular}
\end{center}

From the above list, we see that different diffeomorphism types of 2d effective connected $S^1$-manifolds have different $(b,f,s)$-vectors, together with different $(\bar{b},f,s)$-vectors, hence the claim of the theorem follows. 
\end{proof}

\begin{rmk}
Though the integer $(b,f,s)$-vector or $(\bar{b},f,s)$-vector classifies all the 2d effective connected $S^1$-manifolds, their values are limited to the seven cases.
\end{rmk}

\begin{rmk}
For 2d effective $S^1$-manifolds without boundary, the $(f,s)$-vector is enough to give the classification.
\end{rmk}

\begin{rmk}
The author learned this folklore classification theorem from Audin's book (\cite{Au04} Sec I.3). The numeric version here is just a simple corollary.
\end{rmk}

\subsection{$S^1$-actions on closed 3d manifolds}
The idea of classifying effective $S^1$-actions in dimension 3 is the same as in dimension 2 by listing all the possible equivariant tubular neighbourhoods of non-principal orbits, and then try to patch them together. But one more dimension for the isotropic representations provides a longer list of equivariant tubular neighbourhoods.

\subsubsection{Equivariant tubular neighbourhoods of principal orbits}
For a point $x$ of principal type, its isotropy group is the identity group $\{1\}$ with a trivial isotropic representation $\{1\} \curvearrowright \mathbb{R}^2$. So an equivariant tubular neighbourhood of $S^1\cdot x$ can be written as $S^1 \times_{\{1\}} D = S^1 \times D$, with the $S^1$-action concentrating entirely on the $S^1$ component. So the orbit space of this tubular neighbourhood is $(S^1 \times D)/S^1=S^1/S^1 \times D=D$, a smooth local chart.

\subsubsection{Equivariant tubular neighbourhoods of exceptional orbits}
The union of exceptional orbits will be denoted as $E$. For an exceptional orbit $S^1/\mathbb{Z}_m$ with stabilizer $\mathbb{Z}_m = \{e^{\frac{2\pi k i}{m} }, k = 1,2,\ldots,m\}$, its isotropic representation of $\mathbb{Z}_m$ is 2-dimensional. Such a 2-dimensional effective $\mathbb{Z}_m$-representation could preserve the orientation by rotating:
\[
\mathbb{Z}_m \overset{\text{rotate}}{\curvearrowright} \mathbb{C}: 
\quad e^{\frac{2\pi k i}{m}}\circ z = (e^{\frac{2\pi k i}{m}})^n z
\]
where the orbit invariants $(m,\, n)$, also called Seifert invariants, are coprime positive integers, and $0<n<m$. The resulting equivariant tubular neighbourhood is $S^1\times_{\mathbb{Z}_m}D$, whose orbit space is an orbifold disk
\[(S^1 \times_{\mathbb{Z}_m} D)/S^1 = D/\mathbb{Z}_m\]
where the central orbifold point $pt/\mathbb{Z}_m$ corresponds to the exceptional orbit $S^1/\mathbb{Z}_m$.

\subsubsection{Equivariant tubular neighbourhoods of special exceptional orbits}
Besides rotating, a 2-dimensional effective $\mathbb{Z}_m$-representation could also reverse the orientation by reflecting:
\[
\mathbb{Z}_2 \overset{\text{reflect}}{\curvearrowright} \mathbb{R}^2: 
\quad e^{\pi i}\circ (x,y) = (-x,y)
\]
This case requires the $\mathbb{Z}_m$ to be $\mathbb{Z}_2$. Because of the reverse of orientation, we call such an orbit $S^1/\mathbb{Z}_2$ a \textbf{special exceptional orbit}. The union of all such special exceptional orbits will be denoted as $SE$.

If we use the open square $I\times I = \{(x,y) \mid -1<x,\,y<1\}$ as a neighbourhood in $\mathbb{R}^2$, an equivariant tubular neighbourhood of the special exceptional orbit $S^1/\mathbb{Z}_2$ can be written as $S^1\times_{\mathbb{Z}_2}(I\times I)$, the orbit space by $\mathbb{Z}_2$ of the solid torus $S^1\times (I\times I)$. Note that the reflection $\mathbb{Z}_2 \overset{\text{reflect}}{\curvearrowright} I\times I: 
\quad e^{\pi i}\circ (x,y) = (-x,y)$ only affects the first $I$ component, so we can split the second $I$ component out of the orbit space $S^1\times_{\mathbb{Z}_2}(I\times I)$:
\begin{eqnarray*}
S^1\times_{\mathbb{Z}_2}(I\times I) &=& S^1\times (I\times I)/(e^{i\theta},x,y)\sim(-e^{i\theta},-x,y)\\
&=& \Big(S^1\times I/(e^{i\theta},x)\sim(-e^{i\theta},-x)\Big)\times I = \text{M\"{o}b} \times I
\end{eqnarray*}
where we write M\"{o}b for short of the M\"{o}bius band $S^1\times_{\mathbb{Z}_2}I$. 

Because the set of points with stabilizer $\mathbb{Z}_2$ in the M\"{o}bius band $S^1\times_{\mathbb{Z}_2}I$ is $\text{M\"{o}b}_{(\mathbb{Z}_2)}=S^1\times_{\mathbb{Z}_2}\{0\}=S^1/\mathbb{Z}_2$ a circle, the set of points with stabilizer $\mathbb{Z}_2$ in $\text{M\"{o}b}\times I$ is $(\text{M\"{o}b}\times I)_{(\mathbb{Z}_2)}=S^1/\mathbb{Z}_2\times I$ of dimension 2. Thus, if a 3d $S^1$-manifold $M$ has a special exceptional orbit $S^1/\mathbb{Z}_2$, then the connected component of $M_{(\mathbb{Z}_2)}$ that contains this orbit will be of dimension 2 and is acted freely by $S^1/\mathbb{Z}_2$, hence has to be $S^1/\mathbb{Z}_2 \times S^1$ according the list of 2d $S^1$-manifolds.

Now an equivariant tubular neighbourhood of this torus $S^1/\mathbb{Z}_2 \times S^1$ will be a bundle of M\"{o}bius band over $S^1$, which is actually a product bundle $\text{M\"{o}b} \times S^1$, cf. Raymond \cite{Ra68}. 

Notice that the $S^1$-action concentrates entirely on the component of M\"{o}bius band, so the orbit space is $(\text{M\"{o}b} \times I)/S^1 = \text{M\"{o}b}/S^1 \times I = [0,1)\times I$ with a boundary circle $\{0\}\times S^1$.

\subsubsection{Equivariant tubular neighbourhoods of fixed points}
The set of fixed points will be denoted as $F$. For a fixed point $x$ with stabilizer $S^1$, its isotropic representation is of dimension 3. There is only one such effective $3$-dimensional $S^1$-representation $S^1 \curvearrowright \mathbb{C}\oplus\mathbb{R}$ by acting on the $\mathbb{C}$ component with rotation and acting on the $\mathbb{R}$ component trivially.

So an equivariant tubular neighbourhood of $x$ can be written as $D\times I$, with fixed point set $\{0\} \times I$, an interval. We can continue to glue along this fixed interval to form $S^1$, a connected component of the fixed point set. Now an enlarged equivariant tubular neighbourhood of the fixed circle $S^1$ is going to be a disk bundle over the $S^1$, which is actually a product bundle $D\times S^1$, cf. Raymond \cite{Ra68}. 

Notice that the $S^1$-action concentrates entirely on the $D$ component, so the orbit space is $(D\times S^1)/S^1 = D/S^1 \times S^1 = [0,1)\times S^1$ with a boundary circle $\{0\}\times S^1$.

\subsubsection{Patching: from local to global}
First, we can put all the local discussions into a list

\begin{center}
\begin{tabular}{c c c c c}
\toprule
 & Principal &  Exceptional & Special exceptional & Singular\\
\midrule
Stabilizer $S^1_x$& $\{1\}$ & $\mathbb{Z}_m$ & $\mathbb{Z}_2$& $S^1$\\
\midrule
Isotropic repr & $\{1\}\curvearrowright \mathbb{C}$ &
$\mathbb{Z}_m \overset{\text{rotate}}{\curvearrowright} \mathbb{C}$ &
$\mathbb{Z}_2 \overset{\text{reflect}}{\curvearrowright} \mathbb{R}^2$ &
$S^1 \overset{\text{rotate}}{\curvearrowright} \mathbb{C}\oplus \mathbb{R}$\\
\midrule
Orbit $S^1\cdot x$ & $S^1$ & $S^1/\mathbb{Z}_m$ & $S^1/\mathbb{Z}_2$& $pt $\\
Equiv nbhd & $S^1 \times D $ & $S^1 \times_{\mathbb{Z}_m} D$ & $\text{M\"{o}b} \times I$ & $D\times I$ \\
Orbit space & $D$ & $D/\mathbb{Z}_m$ & $[0,1)\times I$ & $[0,1)\times I$\\
\midrule
Union of orbits & $S^1$ & $S^1/\mathbb{Z}_m$ & $S^1/\mathbb{Z}_2 \times S^1$ & $pt \times S^1 $ \\
Enlarged nbhd & $S^1 \times D $ & $S^1 \times_{\mathbb{Z}_m} D$ & $\text{M\"{o}b} \times S^1$ & $D\times S^1$ \\
Enlarged orbit space & $D$ & $D/\mathbb{Z}_m$ & $[0,1)\times S^1$ & $[0,1)\times S^1$\\
\bottomrule
\end{tabular}
\end{center} 

From the above list, we see that, passing to the orbit space, the local neighbourhood of an exceptional orbit $S^1/\mathbb{Z}_m$ contributes to an orbifold neighbourhood. Both the local neighbourhoods of special exceptional orbits and the local neighbourhoods of fixed circles give rise to half closed, half open annuli with circle boundaries. 

\begin{thm}[Orbit space of closed 3d $S^1$-manifold, \cite{Ra68,OR68}]
For a compact closed 3d effective $S^1$-manifold $M$, the orbit space $M^*=M/S^1$ is a 2d orbifold surface, possibly with boundaries. The orbifold surface $M^*$ has finite number of interior orbifold points with Seifert invariants $\{(m_1,\,n_1),\ldots,(m_r,\,n_r)\}$, and boundary $\partial M^* = F\cup SE/S^1$ coming from the fixed circles and special exceptional orbits.
\end{thm}

To express $M^*=M/S^1$ and its orbifold points into numeric data, let's denote $\epsilon$ as the orientability of the 2d orbit space $M^*=M/S^1$, $g$ the genus, $(f,\,s)$ the numbers of circles formed from union of fixed points and union of special exceptional orbits respectively.

As for the total space $M$, after specifying the neighbourhoods of non-principal orbits, there is an obstruction integer $b$ of finding a cross section over the principal part of the orbit space. The theorem by Orlik and Raymond says that, these invariants completely classify the 3d $S^1$-manifolds, after adding some constraints within these invariants. The following version is taken from Orlik's lecture notes \cite{Or72}.

\begin{thm}[Equivariant classification of closed 3d $S^1$-manifolds, \cite{Ra68,OR68}]
\label{OR}
Let $S^1$ act effectively and smoothly on a closed, connected smooth 3d manifold $M$. Then the orbit invariants
\[
\big\{b;(\epsilon,g,f,s);(m_1,\,n_1),\ldots,(m_r,\,n_r)\big\}
\]
determine $M$ up to equivariant diffeomorphisms, subject to the following conditions
\begin{itemize}
\item[(1)] 
$b=0$, if $f+s>0$\\
$b \in \mathbb{Z}$, if $f+s=0$ and $\epsilon=o$, orientable\\
$b \in \mathbb{Z}_2$, if $f+s=0$ and $\epsilon=n$, non-orientable\\
$b=0$, if $f+s=0$, $\epsilon=n$ and $m_i = 2$ for some $i$
\item[(2)]
$0<n_i<m_i,\,(m_i,\,n_i)=1$ if $\epsilon=o$\\
$0<n_i\leq \frac{m_i}{2},\,(m_i,\,n_i)=1$ if $\epsilon=n$
\end{itemize}
Conversely, any such set of invariants can be realized as a closed 3d manifold with an effective $S^1$-action.
\end{thm}

\begin{rmk}
When $M$ has neither fixed point nor special exceptional orbit, i.e. $f=s=0$, then this is the case of classic Seifert manifolds. 
\end{rmk}

\begin{rmk}
The invariants in $M = \big\{b;(\epsilon,g,f,s);(m_1,\,n_1),\ldots,(m_r,\,n_r)\big\}$ mostly come from the orbit space $M^*=M/S^1$ except the invariant $b$. Therefore the constraint ($b=0$, if $f+s>0$) says that if the orbifold $M^*$ has boundaries, then $M=\big\{b=0;(\epsilon,g,f,s);(m_1,\,n_1),\ldots,(m_r,\,n_r)\big\}$ has only invariants completely determined by the orbifold $M/S^1$ and the assignment of its boundary circles either being fixed points or orbit space of special exceptional orbits.
\end{rmk}

\begin{rmk}
The above classification is up to equivariant diffeomorphisms. But Orlik and Raymond also discussed in certain conditions, more than one $S^1$-actions can appear on the same 3d manifold.
\end{rmk}

For an orientable $S^1$-manifold $M$, the orbit space $M^*=M/S^1$ will be orientable, i.e. $\epsilon = o$, and there will be no special exceptional orbits, i.e. $s=0$. 
\begin{cor}[Classification of closed orientable 3d $S^1$-manifolds, \cite{Ra68,OR68}]
If a closed 3d $S^1$-manifold is oriented and the $S^1$-action preserves the orientation. Then the orbit invariants
\[
\big\{b;(\epsilon=o,g,f,s=0);(m_1,\,n_1),\ldots,(m_r,\,n_r)\big\}
\]
determine $M$ up to equivariant diffeomorphisms, subject to the following conditions
\begin{itemize}
\item[(1)] 
$b=0$, if $f>0$\\
$b \in \mathbb{Z}$, if $f=0$
\item[(2)]
$0<n_i<m_i,\,(m_i,\,n_i)=1$
\end{itemize}
\end{cor} 

Though the widely cited version of the Orlik-Raymond Theorem describes a closed 3d effective $S^1$-manifold $M$ as a set of unlabelled numeric tuples 
\[
\big\{b;(\epsilon,g,f,s);(m_1,\,n_1),\ldots,(m_r,\,n_r)\big\}
\]
Orlik and Raymond actually proved a much stronger version using a set of labelled data.

Instead of simply counting the numbers of boundary circles in $M^*$ from fixed points or special exceptional orbits as $(f,s)$, one can label every boundary circle of $M^*$ and specify whether it comes from fixed points or special exceptional orbits. Thus there are $f$ labelled circles $F_1,\dots,F_f$ coming from fixed points and $s$ labelled circles $SE_1,\dots,SE_s$ coming from special exceptional components.

Likewise, we assemble all the orbifold points $E_1,\ldots,E_r$ of $M^*=M/S^1$ together with the Seifert invariants as $(E_1,\,m_1,\,n_1),\ldots,(E_r,\,m_r,\,n_r)$.
 
Let's denote 
\[
\mathcal{L}=\big\{(E_1,\,m_1,\,n_1),\ldots,(E_r,\,m_r,\,n_r);F_1,\dots,F_f;SE_1,\dots,SE_s\big\}
\]
as the collection of labellings on boundaries and orbifold points of $M^*$. Now, since the orientability $\epsilon$ and the genus $g$ are determined by the orbifold surface $M^*$ itself, the tuple $(M^*,\mathcal{L},b)$ includes all the numeric data in $\{b;(\epsilon,g,f,s);(m_1,\,n_1),\ldots,(m_r,\,n_r)\}$.

\begin{rmk}
Changing the subscripts leads to different labellings. For instance, a pair of Seifert invariants $(m,n)$ can be labelled as $(E_1,\,m,\,n)$, but they can also be labelled as $(E_2,\,m,\,n)$ in another collection of labellings. However, the integer subscripts are just meant to distinguish different orbifold points or different boundary circles, not meant to rank them. 
\end{rmk}

\begin{dfn}
A map $\varphi$ between two labelled orbifold surfaces $(M^*,\mathcal{L},b)$ and $(\bar{M}^*,\bar{\mathcal{L}},\bar{b})$ is a \textbf{labelled orbifold diffeomorphism} if
\begin{itemize}
\item 
$\varphi:M^* \rightarrow \bar{M}^*$ is an orbifold diffeomorphism that extends well to the boundaries
\item
$\varphi:\mathcal{L} \rightarrow \bar{\mathcal{L}}$ respects labellings such that $E_i\overset{\varphi}{\mapsto}\bar{E}_j$ with $(m_i,n_i)=(\bar{m}_j,\bar{n}_j)$, $F_k\overset{\varphi}{\mapsto}\bar{F}_l$, $SE_u\overset{\varphi}{\mapsto}\bar{SE}_v$
\item
$b=\bar{b}$
\end{itemize}
\end{dfn}

\begin{rmk}
The rankings in the labellings $\{F_1,\dots,F_f;SE_1,\dots,SE_s;(x_1,\,m_1,\,n_1),\ldots,(x_r,\,m_r,\,n_r)\big\}$ are not important. So we don't require a labelled orbifold diffeomorphism $\varphi:(M^*,\mathcal{L},b)\rightarrow(\bar{M}^*,\bar{\mathcal{L}},\bar{b})$ to preserve the rankings. For instance, we could have $\varphi(F_1)=\bar{F}_2$, not necessarily $\varphi(F_1)=\bar{F}_1$ nor $\varphi(F_2)=\bar{F}_2$.
\end{rmk}

\begin{rmk}
As we pointed out before, a closed 3d effective $S^1$-manifold $M$ can have different labelled orbit space $(M^*,\mathcal{L}_1,b)$ and $(M^*,\mathcal{L}_2,b)$ by changing subscripts in the labellings. But according to the above definition, these two labelled orbit space are labelled orbifold diffeomorphic to each other simply via the identity map. Therefore, we will say \textbf{the} labelled orbit space $(M^*,\mathcal{L},b)$ with a slight abuse of notation. 
\end{rmk}

Here is the labelled version of classification of 3d closed $S^1$-manifolds, essentially due to Orlik and Raymond \cite{Ra68,OR68}:
\begin{thm}[Labelled classification of closed 3d $S^1$-manifolds, \cite{Ra68,OR68}]
\label{OR-labelled}
Given two compact closed 3d effective $S^1$-manifolds $M$ and $\bar{M}$, let's fix labellings for their orbit spaces as $(M^*,\mathcal{L},b)$ and $(\bar{M}^*,\bar{\mathcal{L}},\bar{b})$ respectively. Then, $M$ and $\bar{M}$ are $S^1$-equivariant diffeomorphic if and only if their labelled orbit spaces $(M^*,\mathcal{L},b)$ and $(\bar{M}^*,\bar{\mathcal{L}},\bar{b})$ are labelled orbifold diffeomorphic.

Moreover, there is a commutative diagram:
\begin{center}
\begin{tikzpicture}[description/.style={fill=white,inner sep=2pt}]
 \matrix (m) [matrix of math nodes, row sep=2em,
 column sep=2em, text height=1.5ex, text depth=0.25ex]
 {M & \bar{M}\\
 (M^*,\mathcal{L},b) & (\bar{M}^*,\bar{\mathcal{L}},\bar{b})\\};
 \path[->] 
 (m-1-1) edge node[above] {$\Phi$} node[below] {$\cong$}(m-1-2)
 (m-1-1) edge node[left] {$\pi$}(m-2-1)
 (m-2-1) edge node[above] {$\varphi$}node[below] {$\cong$}(m-2-2)
 (m-1-2) edge node[right] {$\bar{\pi} $} (m-2-2);
\end{tikzpicture}
\end{center}
that lifts a labelled orbifold diffeomorphism to an $S^1$-diffeomorphism.
\end{thm}

\begin{rmk}
Though never stated explicitly, this version of the Orlik-Raymond theorem is the backbone of Orlik and Raymond's proof of the common version: if two closed, effective $S^1$-manifolds $M^3$ and $\bar{M}^3$ have the same numeric data $\{b;(\epsilon,g,f,s);(m_1,\,n_1),\ldots,(m_r,\,n_r)\}$, then one can find a labelled orbifold diffeomorphism between their 2-dimensional labelled orbit spaces $(M^*,\mathcal{L},b)$ and $(\bar{M}^*,\bar{\mathcal{L}},\bar{b})$ by playing them like 2-dimensional rubber sheets carefully with the orbifold points and boundaries in mind, hence $M^3$ and $\bar{M}^3$ are $S^1$-equivariant diffeomorphic by the laballed version of Orlik-Raymond theorem. 
\end{rmk}

\begin{rmk}
Raymond's idea of proving this theorem is to lift the labelled orbifold diffeomorphism $(M^*,\mathcal{L},b)\overset{\varphi}{\rightarrow} (\bar{M}^*,\bar{\mathcal{L}},\bar{b})$ to an $S^1$-diffeomorphism $M \overset{\Phi}{\rightarrow} \bar{M}$. The lifting is carried out in steps: firstly, lift the correspondence $\varphi:\mathcal{L} \rightarrow \bar{\mathcal{L}}$ to $\Phi: E\cup F \cup SE \rightarrow \bar{E}\cup \bar{F} \cup \bar{SE}$ between non-principal orbits; secondly, extend this map to a tubular neighbourhood of the non-principal orbits; finally, extend this map to all the principal orbits using local cross sections, which actually gives a global $S^1$-diffeomorphism because $b=\bar{b}$.
\end{rmk}

\begin{rmk}
Given a labelled orbifold diffeomorphism $(M^*,\mathcal{L},b)\overset{\varphi}{\rightarrow} (\bar{M}^*,\bar{\mathcal{L}},\bar{b})$, the lifted $S^1$-diffeomorphism $M \overset{\Phi}{\rightarrow} \bar{M}$ is not necessarily unique up to equivariant isotopies. For example, let $P$ be a principal $S^1$-bundle over a closed surface $M_g$ of genus $g$. Any $S^1$-automorphism $\Phi$ on $P$ that induces the identity map on $M_g$ corresponds to a map $f_\Phi: M_g \rightarrow S^1$. Thus, $\Phi$ up to $S^1$-isotopy corresponds to $f_\Phi$ up to homotopy, which is classified by $H^1(M_g,\mathbb{Z})$.
\end{rmk}

\vskip 20pt
\section{$S^1$-actions on 3d manifolds with boundaries}
\vskip 15pt

Let $M^3$ be a compact connected 3d manifold with an effective $S^1$-action that extends compatibly to its non-empty boundary $\partial M$. Combining the classification of $S^1$-actions on closed 2d and 3d manifolds, we can generalize the Orlik-Raymond classification theorem to $S^1$-actions on 3d manifolds with boundaries.

\subsection{Neighbourhoods and orbit spaces}\label{subsec:neighbour}
Similar to the case of 3d $S^1$-manifolds without boundary, we will first give a complete description of the equivariant neighbourhoods.

\subsubsection{Collaring neighbourhoods of boundaries and their orbit spaces}
The Theorem of equivariant collaring neighbourhood says that for any $S^1$-invariant boundary component $B$ in $\partial M$, an equivariant collaring neighbourhood of $B$ in $M$ looks like $B\times [0,1)$, whose orbit space is $B/S^1 \times [0,1)$.

We have seen in the discussion of 2d closed $S^1$-manifolds that there are four of them up to equivariant diffeomorphisms: $\mathbb{T}^2,\, S^2,\,K,\, \mathbb{R}P^2$ which will appear as boundaries of 3d $S^1$-manifolds. 

The non-principal orbits appearing in $\mathbb{T}^2,\, S^2,\,K,\, \mathbb{R}P^2$ are either fixed points or $S^1/\mathbb{Z}_2$ with isotropy representation being a reflection, hence a special exceptional orbit. Therefore, among the union of the non-principal orbits $E\cup F \cup SE$, the boundary $\partial M$ is separated from the exceptional orbits $E$, but could possibly have common points with the fixed points and special exceptional orbits $F \cup SE$.

More explicitly, each boundary component $\mathbb{T}^2$ has an equivariant collaring neighbourhood $\mathbb{T}^2 \times [0,1)$ consisting of only principal orbits. The orbit space $\mathbb{T}^2/S^1 \times [0,1)=S^1 \times [0,1)$ is a half closed, half open annulus with a circle boundary:
\begin{center}
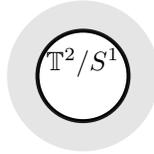
\begin{figure}[H]
\begin{tikzpicture}
\draw[line width=1.5pt]
(0,0) circle (0.6cm) node[label={[yshift=-0.3cm]$\mathbb{T}^2/S^1$}] {};
\path [draw=none, fill=gray, opacity = 0.2, even odd rule] (0,0) circle (0.6cm) (0,0) circle (1cm);
\end{tikzpicture}
\caption{Cylinder $S^1 \times [0,1)$ flattened as annulus}
\label{fig:Annulus}
\end{figure}
\end{center}
where the boundary circle is the orbit space $\mathbb{T}^2/S^1=S^1$.

Each boundary component $S^2$ has an equivariant collaring neighbourhood $S^2 \times [0,1)$ with the two fixed poles $N,\,S$ attached to the two fixed intervals $N\times[0,1),\,S\times[0,1)$ respectively. The orbit space $S^2/S^1 \times [0,1)=[-1,1] \times [0,1)$ is an open manifold with 3 boundaries and 2 corners:
\begin{center}
\begin{figure}[H]
\begin{tikzpicture}
\draw[line width=1.5pt] (-1,1)--(-1,0)-- node[below] {$S^2/S^1$} (1,0)--(1,1); 
\fill[gray,opacity=0.2] (-1,0) rectangle (1,1);
\node[circle, draw, fill=black, inner sep=0pt, minimum size=4.5pt, label={[xshift=-0.3cm, yshift=-0.3cm]$N$}] at (-1,0){};
\node[circle, draw, fill=black, inner sep=0pt, minimum size=4.5pt, label={[xshift=0.3cm, yshift=-0.3cm]$S$}] at (1,0){};
\end{tikzpicture}
\caption{One-side-open rectangle $[-1,1] \times [0,1)$}
\label{fig:Corner1}
\end{figure}
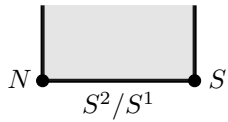
\end{center}
where the the bottom interval is the orbit space $S^2/S^1$, the left and right intervals come from the two fixed intervals $N\times[0,1),\,S\times[0,1)$, the two corner points are the two poles $N,\,S$.

Each boundary component $\mathbb{R}P^2$ has an equivariant collaring neighbourhood $\mathbb{R}P^2 \times [0,1)$ with a fixed point $p$ and the orbit $S^1/\mathbb{Z}_2$ attached to a fixed interval $p\times[0,1)$ and a special exceptional component $S^1/\mathbb{Z}_2 \times [0,1)$ respectively. The orbit space $\mathbb{R}P^2/S^1 \times [0,1)=[0,1] \times [0,1)$ is an open manifold with 3 boundaries and 2 corners:
\begin{center}
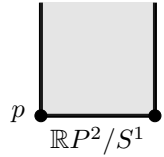
\begin{figure}[H]
\begin{tikzpicture}
\draw[line width=1.5pt] (0,1.5)--(0,0)-- node[below] {$\mathbb{R}P^2/S^1$} (1.5,0)--(1.5,1.5); 
\fill[gray,opacity=0.2] (0,0) rectangle (1.5,1.5);
\node[circle, draw, fill=black, inner sep=0pt, minimum size=4.5pt, label={[xshift=-0.3cm, yshift=-0.3cm]$p$}] at (0,0){};
\node[circle, draw, fill=black, inner sep=0pt, minimum size=4.5pt] at (1.5,0){};
\end{tikzpicture}
\caption{One-side-open rectangle $[0,1] \times [0,1)$}
\label{fig:Corner2}
\end{figure}
\end{center}
where the the bottom interval is the orbit space $\mathbb{R}P^2/S^1$, the left interval is the fixed interval $p\times[0,1)$ with the corner point $p$, and the right interval comes from the orbit space $(S^1/\mathbb{Z}_2 \times [0,1))/S^1$ with the other corner point.

Each boundary component $K$ has an equivariant collaring neighbourhood $K \times [0,1)$ with the two $S^1/\mathbb{Z}_2$-orbits attached to special exceptional components $S^1/\mathbb{Z}_2 \times [0,1)$ respectively. The orbit space $K/S^1 \times [0,1)=[0,\pi] \times [0,1)$ is an open manifold with 3 boundaries and 2 corners:
\begin{center}
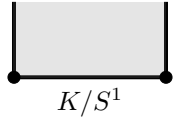
\begin{figure}[H]
\begin{tikzpicture}
\draw[line width=1.5pt] (0,1)--(0,0)-- node[below] {$K/S^1$} (2,0)--(2,1); 
\fill[gray,opacity=0.2] (0,0) rectangle (2,1);
\node[circle, draw, fill=black, inner sep=0pt, minimum size=4.5pt] at (0,0){};
\node[circle, draw, fill=black, inner sep=0pt, minimum size=4.5pt] at (2,0){};
\end{tikzpicture}
\caption{One-side-open rectangle $[0,\pi] \times [0,1)$}
\label{fig:Corner3}
\end{figure}
\end{center}
where the the bottom interval is the orbit space $K/S^1$, the left and right intervals come from the orbit spaces $(S^1/\mathbb{Z}_2 \times [0,1))/S^1$ with corner points.

\subsubsection{Tubular neighbourhoods of non-principal orbits and their orbit spaces}
Equivariant tubular neighbourhoods of an exceptional orbit $S^1/\mathbb{Z}_m$, a fixed circle $S^1$ or a special exceptional component $S^1/\mathbb{Z}_2 \times S^1$ will still be $S^1 \times_{\mathbb{Z}_m}D$, $D\times S^1$ or $\text{M\"{o}b} \times S^1$ respectively, the same as we see in the case of 3d $S^1$-manifolds without boundary. The orbit spaces of these neighbourhoods provide orbifold chart $D/\mathbb{Z}_m$ and annulus charts $[0,1)\times S^1$ for $M/S^1$.

Suppose a fixed connected component $F_i$ has common points with the boundary $\partial M$. As an 1d compact manifold with boundary, $F_i$ has to be a closed interval, denoted as $[0,1]$. So its equivariant tubular neighbourhood will be $D\times [0,1]$, with boundary $(D\times [0,1])\cap \partial M = D\times \{0\} \cup D\times \{1\}$. The orbit space
\[
\big(D\times [0,1]\big)/S^1 = D/S^1 \times [0,1] = [0,1) \times [0,1]
\]
is an open manifold with 3 boundaries and 2 corners:
\begin{center}
\begin{figure}[H]
\begin{tikzpicture}
\draw[line width=1.5pt] (1.5,1.5)--node[above]{$D/S^1$}(0,1.5)--node[left]{$F_i$}(0,0)--node[below]{$D/S^1$}(1.5,0); 
\fill[gray,opacity=0.2] (0,0) rectangle (1.5,1.5);
\node[circle, draw, fill=black, inner sep=0pt, minimum size=4.5pt] at (0,1.5){};
\node[circle, draw, fill=black, inner sep=0pt, minimum size=4.5pt] at (0,0){};
\end{tikzpicture}
\caption{One-side-open rectangle $[0,1) \times [0,1]$}
\label{fig:Corner4}
\end{figure}
\end{center}
where the the left interval is the fixed interval $F_i$, the bottom and top intervals come from orbit space of the boundary $D\times \{0\} \cup D\times \{1\}$.

Similarly, suppose a special exceptional connected component $SE_j$ has common points with the boundary $\partial M$. As a 2d compact principal $S^1/\mathbb{Z}_2$-manifold with boundary, $SE_j$ has to be a cylinder $S^1/\mathbb{Z}_2 \times [0,1]$, according to the classification theorem in dimension 2. So its equivariant tubular neighbourhood will be $\text{M\"{o}b} \times [0,1]$, with boundary $(\text{M\"{o}b}\times [0,1])\cap \partial M = \text{M\"{o}b}\times \{0\} \cup \text{M\"{o}b}\times \{1\}$. The orbit space
\[
\big(\text{M\"{o}b}\times [0,1]\big)/S^1 = \text{M\"{o}b}/S^1 \times [0,1] = [0,1) \times [0,1]
\]
is an open manifold with 3 boundaries and 2 corners:
\begin{center}
\begin{figure}[H]
\begin{tikzpicture}
\draw[line width=1.5pt] (1.5,1.5)--node[above]{$\text{M\"{o}b}/S^1$}(0,1.5)--node[left]{$SE_j/S^1$}(0,0)--node[below]{$\text{M\"{o}b}/S^1$}(1.5,0); 
\fill[gray,opacity=0.2] (0,0) rectangle (1.5,1.5);
\node[circle, draw, fill=black, inner sep=0pt, minimum size=4.5pt] at (0,1.5){};
\node[circle, draw, fill=black, inner sep=0pt, minimum size=4.5pt] at (0,0){};
\end{tikzpicture}
\caption{One-side-open rectangle $[0,1) \times [0,1]$}
\label{fig:Corner5}
\end{figure}
\end{center}
where the the left interval comes from the orbit space $SE_j/S^1$, the bottom and top intervals come from orbit space of the boundary $\text{M\"{o}b}\times \{0\} \cup \text{M\"{o}b}\times \{1\}$.

\subsubsection{Orbit spaces of 3d $S^1$-manifolds with boundaries}
Using the discussion of local orbit spaces, we have the following theorem about orbit space of 3d $S^1$-manifolds with boundaries:

\begin{prop}[Orbit space of 3d $S^1$-manifold]
For a compact connected 3d effective $S^1$-manifold $M$ possibly with boundary, the orbit space $M^*=M/S^1$ is a 2d orbifold surface, possibly with corners. The orbifold surface $M^*$ has finite number of interior orbifold points $E/S^1$ with Seifert invariants $\{(m_1,\,n_1),\ldots,(m_r,\,n_r)\}$. It has boundary $\partial M^* = F\cup SE/S^1\cup (\partial M)/S^1$ and corner points $\partial^2 M^* = (F\cup SE/S^1)\cap (\partial M)/S^1$.
\end{prop}
\begin{proof}
From the local analysis, it is clear that $M^*=M/S^1$ is a 2d orbifold surface possibly with corners. Any orbifold point $pt/\mathbb{Z}_m$ is an interior point, because it is the only orbifold point in its neighbourhood $D/\mathbb{Z}_m$. The boundary $\partial M^*$ comes from the fixed point set $F$, special exceptional orbits $SE$ and the boundary $\partial M$. The corner points of the $M^*$ appears when a fixed component or a special exceptional component in $M$ meets the boundary $\partial M$.
\end{proof}

\begin{rmk}
The notation $\partial^2 M^*$ for corners is because in dimension 2, corners are boundary of the boundary. For general theory of manifolds with corners, see Joyce \cite{Jo12}.
\end{rmk}

\subsection{Labellings on 2d orbifold surface with corners}
Similar to the labelling procedure for 2d orbifold surface without corners, we will give labellings to 2d orbifold surface $M^*=M/S^1$, possibly with corners, as follows:

\begin{itemize}
\item
$(E_1,\,m_1,\,n_1),\ldots,(E_r,\,m_r,\,n_r)$: the orbifold points
\item
$F_1,\dots,F_{f}$: circle boundaries of $M^*$ that come from fixed components in $M$ not touching $\partial M$
\item
$F_{f+1},\dots,F_{f_0}$: interval boundaries of $M^*$ that come from fixed components in $M$ touching $\partial M$
\item
$SE_1,\dots,SE_{s}$: circle boundaries of $M^*$ that come from special exceptional components in $M$ not touching $\partial M$
\item 
$SE_{s+1},\dots,SE_{s_0}$: interval boundaries of $M^*$ that come from special exceptional components in $M$ touching $\partial M$
\item
$T_1,\dots,T_t$: circle boundaries of $M^*$ that are the orbit spaces of torus boundaries of $M$
\item 
$SP_1,\dots,SP_{s_p}$: interval boundaries of $M^*$ that are the orbit spaces of sphere boundaries of $M$
\item
$K_1,\dots,K_k$: interval boundaries of $M^*$ that are the orbit spaces of Klein bottle boundaries of $M$
\item
$RP_1,\dots,RP_{r_p}$: interval boundaries of $M^*$ that are the orbit spaces of projective plane boundaries of $M$
\item
$V^F_1,\dots,V^F_{v_f}$: corner points of $M^*$ that come from the fixed points of $\partial M$, i.e. $F \cap \partial M$
\item
$V^S_1,\dots,V^S_{v_s}$: corner points of $M^*$ that come from the special exceptional orbits of $\partial M$, i.e. $SE \cap \partial M$
\end{itemize}

Those labellings are not isolated from one another. In fact, if we denote
\begin{align*}
    \mathcal{V} &= \{V^F_1,\dots,V^F_{v_f}\}\cup\{V^S_1,\dots,V^S_{v_s}\}\\
    \mathcal{E} &= \{F_{f+1},\dots,F_{f_0}\}\cup \{SE_{s+1},\dots,SE_{s_0}\}\cup \{SP_1,\dots,SP_{s_p}\}\cup \{K_1,\dots,K_k\}\cup\{RP_1,\dots,RP_{r_p}\}
\end{align*}
Then we get a labelled graph $\mathcal{G}=\{\mathcal{V},\mathcal{E}\}$, with the set of labelled vertices $\mathcal{V}$ and the set of labelled edges $\mathcal{E}$.

\begin{prop}
The labelled graph $\mathcal{G}=\{\mathcal{V},\mathcal{E}\}$ has the following properties:
\begin{itemize}
\item[(1)]
$\mathcal{G}$ is a union of cycle graphs. Or equivalently, each vertex is the endpoint of exactly two edges, one resulting from $F\cup SE$, the other resulting from $\partial M$.
\item[(2)]
There are numeric relations:
\begin{align*}
v_f &= 2(f_0-f)= 2s_p+r_p\\
v_s &= 2(s_0-s)= 2k+r_p    
\end{align*}
\end{itemize}
\end{prop}
\begin{proof}
For any vertex $V^F_i \in F \cap \partial M$, it joins two intervals. Among these two intervals, one comes from the fixed set $F$; the other one comes from $\partial M$, either being a sphere or a projective plane. 

On the one hand, a fixed interval $F_{f_0+j}$ has two end-points both being in $\{V^F_1,\dots,V^F_{v_f}\}$. Any two fixed intervals don't share fixed end-points, otherwise they will form a single fixed interval. So we get $v_f = 2(f_0-f)$.

On the other hand, among the boundary components of $M^3$, a sphere boundary has two fixed points which are the two poles; a projective plane boundary has exactly one fixed points. All these boundary components are separated from one another. So we get $v_f = 2s_p+r_p$.

For vertices $\{V^S_1,\dots,V^S_{v_s}\}$, the argument and computation is similar.
\end{proof}

As a simple corollary, the numeric relations reveal the parity information. 
\begin{cor}\label{cor:Parity}
Both $v_f$ and $v_s$, the numbers of corner points of two different types, are even. The number of $\mathbb{R}P^2$ boundaries of $M^3$, denoted as $r_p$, is also even. Actually, the number of $RP$-edges in every cycle component of $\mathcal{G}$ is even.
\end{cor}

Now, we can gather all these labellings into a system for $M^*=M/S^1$ as
\[
\mathcal{L}=\big\{(E_1,\,m_1,\,n_1),\ldots,(E_r,\,m_r,\,n_r);F_1,\dots,F_{f};SE_1,\dots,SE_{s};T_1,\dots,T_t;\mathcal{G}\big\}
\]

\begin{dfn}
Let $(M^*,\mathcal{L})$ and $(\bar{M}^*,\bar{\mathcal{L}})$ be two labelled orbifold surfaces with corners. A map $\varphi$ between them is a \textbf{labelled orbifold diffeomorphism} if
\begin{itemize}
\item 
$\varphi:M^* \rightarrow \bar{M}^*$ is an orbifold diffeomorphism that extends well to the boundaries and corners
\item
$\varphi:\mathcal{L} \rightarrow \bar{\mathcal{L}}$ respects the types of labellings and the graph structures
\end{itemize}
\end{dfn}

\begin{rmk}
When an $S^1$-manifold $M^3$ does not have boundary, we use $(M^*,\mathcal{L},b)$ for its labelled orbit space. When $M^3$ does have boundary, we don't need the numeric $b$, for a reason that will be explained shortly. However, in the end, we can include the numeric $b$ back by setting $b=0$ for $S^1$-manifolds with boundaries.  
\end{rmk}

\subsection{Classification of 3d $S^1$-manifolds with boundaries}
Our idea of classifying 3d $S^1$-manifolds with boundaries is to cap off the boundary and then pass to the case of $S^1$-manifolds without boundaries.

Let's keep in mind the following standard equivariant fillings for $\mathbb{T}^2,\, S^2,\,K$
\begin{eqnarray*}
\mathbb{T}^2 &=& S^1 \times S^1 = \partial (S^1 \times D^2)\\
S^2 &=& \partial D^3\\
K &=& S^1 \times_{\mathbb{Z}_2} S^1 = \partial (S^1 \times_{\mathbb{Z}_2} D^2)\\
\end{eqnarray*}
and the standard equivariant filling for a pair of $\mathbb{R}P^2$:
\[
\mathbb{R}P^2 \cup \mathbb{R}P^2 = \partial (\mathbb{R}P^2\times I)
\]

\begin{prop}[Capping off boundary]
\label{capOff}
Let $M$ be a 3d compact, connected, effective $S^1$-manifold with boundary, together with a labelled orbit space $(M^*,\mathcal{L}_M)$. We can cap off the boundary $\partial M$ equivariantly using the standard fillings to form a compact connected 3d effective $S^1$-manifold $N$ without boundary, together with a labelled orbit space $(N^*,\mathcal{L}_N,b=0)$. 
\end{prop}
\begin{proof} First, We recall from the Theorem of equivariant collaring neighbourhood that every $S^1$-invariant boundary $B$ has an equivariant neighbourhood diffeomorphic to $B\times [0,1)$. Then, we will prove our capping-off theorem in three steps:
\begin{description}
\item[Step 1] Cap off $\mathbb{T}^2,\, S^2,\,K$ individually

For any boundary component $B$ of one of the types $\mathbb{T}^2,\, S^2,\,K$ with collaring neighbourhood $B\times [0,1)$, let $U$ be the corresponding standard filling among the options $S^1 \times D^2,\,D^3,\,S^1 \times_{\mathbb{Z}_2} D^2$. We can choose the identity map $Id_B:\partial U \rightarrow B$, then cap off $B$ by gluing $U$ with $M$ along $B$ to form $M \cup_{Id_B} U$.

\item[Step 2] Cap off $\mathbb{R}P^2$ in pairs

According to the Corollary \ref{cor:Parity}, the number of boundaries of type $\mathbb{R}P^2$ is even. So we can divide the collection of $\mathbb{R}P^2$ boundaries in pairs. For any one such pair $\mathbb{R}P_0^2 \cup \mathbb{R}P_1^2$ with collaring neighbourhood $(\mathbb{R}P_0^2 \cup \mathbb{R}P_1^2) \times[0,1)$ in $M$, we can choose the identity map $Id_{\mathbb{R}P^2}:\partial(\mathbb{R}P^2 \times[0,1])\rightarrow \mathbb{R}P_0^2 \cup \mathbb{R}P_1^2$, then cap off the pair by gluing $\mathbb{R}P^2 \times[0,1]$ with $M$ to form $M \cup_{Id_{\mathbb{R}P^2}} (\mathbb{R}P^2 \times[0,1])$.

\item[Step 3] Modify the capping so that $b=0$

After the two steps of capping, we now get a compact, connected, effective $S^1$-manifold $\tilde{N}$ without boundary. According to the labelled version of Orlik-Raymond Theorem, $\tilde{N}$ has a labelled orbit space $(\tilde{N},\tilde{\mathcal{L}},\tilde{b})$. Here $\tilde{b}$ is an integer obstruction of finding cross section for principal orbits, i.e. the Euler number for the principal part of $\tilde{N}$. It can be modified by changing local cross sections, see any one of the papers \cite{Ra68,OR68}.

If $\tilde{b}=0$, then we just take $\tilde{N}$ as our target. Otherwise if $\tilde{b}\not=0$, then according to the Orlik-Raymond theorem, there will be no fixed points nor special exceptional orbits in $\tilde{N}$, i.e. $F_{\tilde{N}} \cup SE_{\tilde{N}} = \varnothing$, and so is $F_M \cup SE_M = \varnothing$ in $M$.

This means that $M$ can only have $\mathbb{T}^2$ boundaries that consists of principal orbits. To modify $\tilde{N}$, we will pick any $\mathbb{T}^2$ boundary component of $M$ with a collaring neighbourhood $\mathbb{T}^2\times [0,1)$ and the identity map $Id: \mathbb{T}^2 \rightarrow \mathbb{T}^2$ that we use in step 1 when we try to cap off the $\mathbb{T}^2$. 

Now instead of using the identity map $Id: \mathbb{T}^2 \rightarrow \mathbb{T}^2$ to cap off the chosen $\mathbb{T}^2$, we will use a twisted $S^1$-diffeomorphism 
\[
Id_{\tilde{b}}: \mathbb{T}^2 \rightarrow \mathbb{T}^2 : (e^{i\theta_1},e^{i\theta_2}) \mapsto (e^{i(\theta_1+\tilde{b}\theta_2)},e^{i\theta_2})
\]

If we keep the capping of the other $\mathbb{T}^2$ boundary components of $M$ the same as in step 1, and modify the capping of the chosen $\mathbb{T}^2$ by using either $Id_{\tilde{b}}$ or $Id_{\text{-}\tilde{b}}$ to glue $S^1\times D$ with $M$, then we will get a new manifold without boundary and $b=\tilde{b}\pm\tilde{b}=0\,\mbox{or}\,2\tilde{b}$. Let $N$ be the one that has $b=0$, then we get the target manifold.
\end{description}
\end{proof}

\begin{rmk}
At the step 2 of capping off $\mathbb{R}P^2$ in pairs, the choice of pairs can be arbitrary, so the constructed $N$ is not unique up to $S^1$-diffeomorphism. But if $M$ does not have $\mathbb{R}P^2$ boundaries, then the $N$ is unique up to $S^1$-diffeomorphism.
\end{rmk}

Now we can prove a labelled version of the classification of 3d $S^1$-manifolds with boundary.

\begin{thm}[Labelled classification of 3d $S^1$-manifolds with boundary]
\label{classifyA}
Let $M$ and $\bar{M}$ be two compact connected 3d effective $S^1$-manifolds with boundary, together with labelled orbit spaces $(M^*,\mathcal{L})$ and $(\bar{M}^*,\bar{\mathcal{L}})$ respectively. Then, $M$ and $\bar{M}$ are $S^1$-equivariant diffeomorphic if and only if their labelled orbit spaces $(M^*,\mathcal{L})$ and $(\bar{M}^*,\bar{\mathcal{L}})$ are labelled orbifold diffeomorphic.

Moreover, there is a commutative diagram:
\begin{center}
\begin{tikzpicture}[description/.style={fill=white,inner sep=2pt}]
 \matrix (m) [matrix of math nodes, row sep=2em,
 column sep=2em, text height=1.5ex, text depth=0.25ex]
 {M & \bar{M}\\
 (M^*,\mathcal{L}) & (\bar{M}^*,\bar{\mathcal{L}})\\};
 \path[->] 
 (m-1-1) edge node[above] {$\Phi$} node[below] {$\cong$}(m-1-2)
 (m-1-1) edge node[left] {$\pi$}(m-2-1)
 (m-2-1) edge node[above] {$\varphi$}node[below] {$\cong$}(m-2-2)
 (m-1-2) edge node[right] {$\bar{\pi} $} (m-2-2);
\end{tikzpicture}
\end{center}
that lifts a labelled orbifold diffeomorphism to an $S^1$-diffeomorphism.
\end{thm}
\begin{proof}
If $M$ and $\bar{M}$ are $S^1$-diffeomorphic, then any $S^1$-diffeomorphism between them induces a labelled orbifold diffeomorphism between their orbit spaces.

Conversely, if $\varphi: (M^*,\mathcal{L})\rightarrow (\bar{M}^*,\bar{\mathcal{L}})$ is a labelled orbifold diffeomorphism. Then we can apply the Capping-Off procedure to $M$ and $\bar{M}$ simultaneously to construct $N$ and $\bar{N}$ without boundary, together with labelled orbit spaces $(N^*,\mathcal{L}_N,b=0)$ and $(\bar{N}^*,\mathcal{L}_{\bar{N}},b=0)$.

If we keep track of the capping-off procedure at the level of labelled orbit spaces,  the $\varphi: (M^*,\mathcal{L})\overset{\cong}{\longrightarrow} (\bar{M}^*,\bar{\mathcal{L}})$ is extended to $\psi: (N^*,\mathcal{L}_N,b=0)\overset{\cong}{\longrightarrow} (\bar{N}^*,\mathcal{L}_{\bar{N}},b=0)$ which can be lifted to be an $S^1$-diffeomorphism $\Psi: N \overset{\cong}{\longrightarrow} \bar{N}$, according to the labelled version of Orlik-Raymond theorem.

Finally, if we restrict the $\Psi$ to $M$, then it gives an $S^1$-diffeomorphism $\Psi|_M: M \overset{\cong}{\longrightarrow} \bar{M}$.
\end{proof}

The unlabelled classification theorem follows from the labelled version and the boundary-less version easily. Using the previous notational system, we denote $\epsilon$ as the orientability of the 2d orbit space, $g$ the genus, $(f,s)$ the numbers of fixed components and special exceptional components not touching $\partial M$ in $M$, $t$ the number of $\mathbb{T}^2$ boundaries of $M$, $\mathcal{G}$ the labelled graph formed from the remaining fixed components, special exceptional components and $S^2,\,K,\,\mathbb{R}P^2$ boundaries of $M$.

\begin{cor}[Classification of 3d $S^1$-manifolds with boundary]
Let $S^1$ act effectively and smoothly on a compact, connected 3d manifold $M$ with boundary. Then the orbit invariants
\[
\big\{(\epsilon,g,f,s,t);(m_1,\,n_1),\ldots,(m_r,\,n_r);\mathcal{G}\big\}
\]
consisting of numeric data and a collection of labelled cycle graphs, 
determine $M$ up to equivariant diffeomorphisms, subject to the following conditions:
\begin{itemize}
\item[(1)]
$0<n_i<m_i,\,(m_i,\,n_i)=1$ if $\epsilon=o$, orientable\\
$0<n_i\leq \frac{m_i}{2},\,(m_i,\,n_i)=1$ if $\epsilon=n$, non-orientable
\item[(2)] For the labelled graph $\mathcal{G}$ with
\begin{align*}
\mathcal{V} &= \{V^F_1,\dots,V^F_{v_f}\}\cup\{V^S_1,\dots,V^S_{v_s}\}\\
\mathcal{E} &= \{F_{f+1},\dots,F_{f_0}\}\cup \{SE_{s+1},\dots,SE_{s_0}\}\cup \{SP_1,\dots,SP_{s_p}\}\cup \{K_1,\dots,K_k\}\cup\{RP_1,\dots,RP_{r_p}\}
\end{align*}
Each $V^F$-vertex is the endpoint of two edges, one edge of the type $F$ and the other edge of either the type $SP$ or $RP$.\\
Each $V^S$-vertex is the endpoint of two edges, one edge of the type $SE$ and the other edge of either the type $K$ or $RP$.\\
The edges of types $F$ and $SE$ connect two $V^F$'s and two $V^S$'s respectively.\\
The edges of types $SP,K,RP$ connect two $V^F$'s, two $V^S$'s, one $V^f$ and one $V^S$ respectively.
\end{itemize}
\end{cor}

If we add the $b=0$ to the case of 3d $S^1$-manifolds with boundary, we can synthesize the cases with or without boundary into one single case:

\begin{thm}[Classification of 3d $S^1$-manifolds]
\label{classifyB}
Let $S^1$ act effectively and smoothly on a compact, connected 3d manifold $M$, possibly with boundary. Then the orbit invariants
\[
\big\{b;(\epsilon,g,f,s,t);(m_1,\,n_1),\ldots,(m_r,\,n_r);\mathcal{G}\big\}
\]
consisting of numeric data and a collection of labelled cycle graphs, 
determine $M$ up to equivariant diffeomorphisms, subject to the following conditions:
\begin{itemize}
\item[(1)] 
$b=0$, if $f+s+t>0$ or $\mathcal{G}\not = \varnothing$\\
$b \in \mathbb{Z}$, if $M$ is orientable\\
$b \in \mathbb{Z}_2$, if $M$ is non-orientable\\
$b=0$, if $f+s+t=0,\,\mathcal{G}=\varnothing,\,\epsilon=n$, and $m_i=2$ for some $i$
\item[(2)]
$0<n_i<m_i,\,(m_i,\,n_i)=1$ if $\epsilon=o$\\
$0<n_i\leq \frac{m_i}{2},\,(m_i,\,n_i)=1$ if $\epsilon=n$
\item[(3)] For the labelled graph $\mathcal{G}$ with
\begin{align*}
\mathcal{V} &= \{V^F_1,\dots,V^F_{v_f}\}\cup\{V^S_1,\dots,V^S_{v_s}\}\\
\mathcal{E} &= \{F_{f+1},\dots,F_{f_0}\}\cup \{SE_{s+1},\dots,SE_{s_0}\}\cup \{SP_1,\dots,SP_{s_p}\}\cup \{K_1,\dots,K_k\}\cup\{RP_1,\dots,RP_{r_p}\}
\end{align*}
Each $V^F$-vertex is the endpoint of two edges, one edge of the type $F$ and the other edge of either the type $SP$ or $RP$.\\
Each $V^S$-vertex is the endpoint of two edges, one edge of the type $SE$ and the other edge of either the type $K$ or $RP$.\\
The edges of types $F$ and $SE$ connect two $V^F$'s and two $V^S$'s respectively.\\
The edges of types $SP,K,RP$ connect two $V^F$'s, two $V^S$'s, one $V^f$ and one $V^S$ respectively.
\end{itemize}
Conversely, any such set of invariants can be realized as a 3d manifold with an effective $S^1$-action.
\end{thm}
\begin{proof}
We have proved the classification part. So we only need to address the realization part. 

Now suppose we are given a set of orbit invariants
\[
\big\{b;(\epsilon,g,f,s,t);(m_1,\,n_1),\ldots,(m_r,\,n_r);\mathcal{G}\big\}
\]
subject to the mentioned conditions. 

If $t=0$ and $\mathcal{G}=\varnothing$, then the orbit invariants
$\big\{b;(\epsilon,g,f,s);(m_1,\,n_1),\ldots,(m_r,\,n_r)\big\}$
are the ones used in the Orlik-Raymond Theorem \ref{OR}, and hence realizable as a closed 3d $S^1$-manifold.

If $t\not=0$ or $\mathcal{G}\not=\varnothing$, then $b=0$ by the condition (1), and the orbit invariants $\big\{b=0;(\epsilon,g,f,s,t);(m_1,\,n_1),\ldots,(m_r,\,n_r);\mathcal{G}\big\}$ presumably come from a 3d $S^1$-manifold with boundary. The realization process takes place in three steps:
\begin{description}
\item[Step 1] Realizing the 2d orbit space

We can first choose a compact surface with orientability $\epsilon$, genus $g$, $f+s+t$ circle boundaries, $r$ interior orbifold points with Seifert invariants $(m_1,\,n_1),\ldots,(m_r,\,n_r)$, and $\mathcal{G}$ as the remaining boundary with corners. We then label $f$ circle boundaries as $F$, $s$ circle boundaries as $SE$, and the remaining $t$ circle boundaries as $T$.
	
\item[Step 2] Capping off the $t,\,\mathcal{G}$ and realizing a closed 3d $S^1$-manifold 

Notice that the capping-off procedure in the Proposition \ref{capOff} also makes sense at the level of orbit space. More specifically, we can cap off the $f$ fixed circle boundary with disks, replace the edges $SP$ and $K$ in $\mathcal{G}$ by edges $F$ and $SE$ respectively and remove the corner points. We can also sew up two $\mathbb{R}P$ edges in $\mathcal{G}$ hence get them removed. After the capping-off of $t,\,\mathcal{G}$, we get a new set of orbit invariants $\big\{b=0;(\epsilon,g',f',s',t=0);(m_1,\,n_1),\ldots,(m_r,\,n_r);\mathcal{G}=\varnothing \big\}$, which is realizable as a closed 3d $S^1$-manifold.

\item[Step 3] Realizing the 3d $S^1$-manifold with boundary

Using the closed 3d $S^1$-manifold from the Step 2, we can precisely reverse the capping-off procedure by removing $t$ solid tori of principal orbits to produce $t$ torus boundaries, removing $D^3$ near fixed points, $S^1\times_{\mathbb{Z}_2}D^2$ near special exceptional orbits and $\mathbb{R}P^2 \times [0,1]$ between fixed points and special exceptional orbits to produce boundaries of types $S^2,K,\mathbb{R}P^2$. Now $g,f,s,t,\mathcal{G}$ are recovered in the orbit space after the reverse procedure. Hence the orbit invariants $\big\{b=0;(\epsilon,g,f,s,t);(m_1,\,n_1),\ldots,(m_r,\,n_r);\mathcal{G}\big\}$ are realized as a 3d $S^1$-manifold with boundary.
\end{description}  
\end{proof}

\begin{rmk}
As we have seen in the above proof, a 3d $S^1$-manifold is closed if and only if $t=0$ and $\mathcal{G}=\varnothing$.
\end{rmk}

\vskip 20pt
\section{Equivariant cohomology of 3d $S^1$-manifolds}
\vskip 15pt

The classification of 3d $S^1$-manifolds (possibly with boundaries) in terms of numeric invariants and graphs gives us an $S^1$-equivariant stratification of every such manifold and enables us to calculate all kinds of topological data. For example, the fundamental groups, ordinary homology and cohomology with $\mathbb{Z}$ or $\mathbb{Z}_p$ coefficients have been computed extensively for closed 3d $S^1$-manifolds in literature \cite{JN83, BHZZ00, BLPZ03, BZ03}, and now can be generalized to 3d $S^1$-manifolds with boundaries, using our classification Theorem \ref{classifyB}. But not much has been discussed for $S^1$-equivariant cohomology, which is the goal of current section.

In the following subsections, we will first prove our core Theorem \ref{equivCohom1} in full generality. When we explore more delicate computational invariants, we will try to keep the presentation of results in a manageable way but perhaps with a slight loss of generality. 

\subsection{Some basic facts about equivariant cohomology}
In this paper, the coefficient of cohomology will always be $\mathbb{Q}$. For a group action of $G$ on $M$, the equivariant cohomology ring is defined using the Borel construction $H^*_G (M) = H^*(EG\times_G M)$, where $H^*(-)$ is the ordinary simplicial cohomology theory, $EG$ is the universal principal $G$-bundle and $EG\times_G M$ is the associated bundle with fibre $M$. The pull-back $\pi^*:H^*_G (pt)\longrightarrow H^*_G (M)$ of the trivial map $\pi:M\longrightarrow pt$ gives $H^*_G (M)$ a module structure of the ring $H^*_G (pt)$. 

In general, the equivariant cohomology $H^*_G (M)$ is not the same as the ordinary cohomology $H^*(M/G)$ of the orbit space $M/G$. If we choose any fibre inclusion $\iota: M \rightarrow EG\times M$ and pass to the orbit spaces $\bar{\iota}:M/G \rightarrow EG\times_G M$, then the pull-back $\bar{\iota}^*: H^*_G (M)=H^*(EG\times_G M)\rightarrow H^*(M/G)$ gives a natural map between $H^*_G (M)$ and $H^* (M/G)$.

We will need some basic facts to compute equivariant cohomology, see any of the expository surveys (\cite{Go,Ty05}) for details.

The first set of facts is about equivariant cohomology of homogeneous spaces, i.e. spaces with single orbits:
\begin{fact}
Let $G$ be a compact Lie group, and $H$ a closed Lie subgroup. Denote $BG=EG/G$ and $BH=EH/H$ for the classifying space of $G$-bundles and $H$-bundles respectively. Then,
\begin{itemize}
\item
$H^*_G(pt)=H^*(EG/G)=H^*(BG)$
\item
$H^*_G(G/H)=H^*_H(pt)=H^*(BH)$
\end{itemize}
\end{fact}

The second set of facts is about equivariant cohomology of extremal types of group actions:
\begin{fact}
Let a compact Lie group $G$ act on a compact manifold $M$.
\begin{itemize}
\item
If the action $G\curvearrowright M$ is free, then $H^*_G(M)=H^*(M/G)$.
\item
If the action $G\curvearrowright M$ is trivial, then $H^*_G(M)=H^*(M)\otimes H^*_G(pt)$.
\end{itemize}
\end{fact}

In particular, when $G=S^1$, there are three types of orbits: $S^1,\,S^1/\mathbb{Z}_m,\,S^1/S^1$. For a principal orbit, $H^*_{S^1}(S^1)=H^*(pt)$. For an exceptional orbit $S^1/\mathbb{Z}_m$, the classifying space $B\mathbb{Z}_m=S^{\infty}/\mathbb{Z}_m$ is the infinite Lens space with cohomology in $\mathbb{Q}$-coefficient the same as $H^*(pt)$. For a fixed point $S^1/S^1$, the classifying space $BS^1=\mathbb{C}P^\infty$ is the infinite projective space with cohomology $\mathbb{Q}[u]$ a polynomial ring, where the parameter $u$ is the generator of $H^2(\mathbb{C}P^1)$ in degree 2.
\begin{center}
\begin{tabular}{c c c c}
\toprule
 & Principal orbit & Exceptional orbit & Singular orbit\\
\midrule
Orbit $\mathcal{O}$ & $S^1$ & $S^1/\mathbb{Z}_m$ & $S^1/S^1$\\
$H^*_{S^1}(\mathcal{O},\mathbb{Q})$ & $H^*(pt,\mathbb{Q})$ & $H^*(pt,\mathbb{Q})$ & $\mathbb{Q}[u]$\\
\bottomrule
\end{tabular}
\end{center}

The third set of facts enables us to compute equivariant cohomology by deforming, cutting and pasting, similar to the computation in ordinary cohomology:

\begin{fact}
Let $U_1,\,U_2$ be two $G$-spaces, and $A,\,B$ be two $G$-subspaces of a $G$-space $X$.
\begin{description}
\item[Homotopy invariance]
If $\varphi: U_1 \overset{\simeq}{\longrightarrow} U_2$ is a $G$-homotopic equivalence, then $\varphi^*: H^*_G(U_2) \overset{\cong}{\longrightarrow} H^*_G(U_1)$ is an isomorphism.
\item[Mayer-Vietoris sequence]
If $X=A^\circ \cup B^\circ$ is the union of interiors of $A$ and $B$, then there is a long exact sequence:
\[
\cdots \longrightarrow H^i_G(X) \longrightarrow H^i_G(A)\oplus H^i_G(B) \longrightarrow H^i_G(A\cap B) \overset{\delta}{\longrightarrow} H^{i+1}_G(X) \longrightarrow \cdots
\]
\end{description} 
\end{fact}

\begin{rmk}
Besides the Borel model of equivariant cohomology, there are also Cartan model and Weil model (cf. Guillemin-Sternberg \cite{GS99}) using equivariant de Rham theory. In this paper, we prefer the Borel model because the homotopy invariance and Mayer-Vietoris sequence are more natural for Borel model, from the topological rather than the differential point of view. 
\end{rmk}

The fourth set of facts deals with equivariant cohomology of product spaces:
\begin{fact}
Let $G \curvearrowright M$ and $H \curvearrowright N$ be two group actions on manifolds. Then, for the product action $G\times H \curvearrowright M\times N$, we get
\[
H^*_{G\times H}(M\times N)=H^*_{G}(M)\otimes H^*_{H}(N)
\]
Especially, for the action $G \curvearrowright M\times N$ where $N$ is acted by $G$ trivially, we get
\[
H^*_{G}(M\times N)=H^*_{G}(M)\otimes H^*(N)
\]
\end{fact}

\subsection{A short exact sequence}
Let $S^1$ act effectively on a compact connected 3d manifold $M$, possibly with boundary. We will compute the equivariant cohomology group $H^*_{S^1}(M,\mathbb{Q})$ by cutting and pasting, with the help of the classification theorem from previous sections.

As we have seen from the previous computation of $H^*_{S^1}(\mathcal{O})$ for any $S^1$-orbit $\mathcal{O}$. The $S^1$-equivariant cohomology in $\mathbb{Q}$ coefficient does not distinguish principal orbit $S^1$ from exceptional orbit $S^1/\mathbb{Z}_m$ or special exceptional orbit $S^1/\mathbb{Z}_2$. However, there is big difference between the $S^1$-equivariant cohomology of fixed point and non-fixed orbit.

If a 3d $S^1$-manifold $M$ does not have fixed points, we would hope that its $S^1$-equivariant cohomology is the ordinary cohomology of the orbit space $M/S^1$. Actually, a more general statement is true due to Satake \cite{Sa56}. The version here is taken from Duistermaat's lecture notes \cite{Du94}.

\begin{dfn}
An action of a Lie group $G$ on a manifold $M$ is \textbf{locally free}, if for any $x \in M$, the isotropy group $G_x$ is finite.
\end{dfn}

\begin{thm}[Satake \cite{Sa56}]
If a compact Lie group $G$ acts locally freely on a compact manifold $M$, then $M/G$ is an orbifold, and $H^*_G(M,\mathbb{R})=H^*(M/G,\mathbb{R})$.
\end{thm}

We can certainly apply the Theorem of Satake to our special case of $S^1$-actions. However, there is a subtlety in Satake's definition of $H^*(M/S^1,\mathbb{R})$ for the orbifold $M/S^1$ in terms of {\em{orbifold}} differential forms (cf. \cite{Sa56,Du94}). Moreover, because of the use of differential forms, the above theorem is originally stated for $\mathbb{R}$-coefficients not for $\mathbb{Q}$-coefficients.  

In our definition of $H^*(M/S^1, \mathbb{Q})$, we will simply use the ordinary simplicial cohomology for the topological space $M/S^1$ by forgetting its orbifold structure. 

\begin{prop}\label{equivCohom0}
Let $S^1$ act effectively on a compact connected 3d manifold $M$, possibly with boundary. If $M$ does not have fixed points, then $H^*_{S^1}(M,\mathbb{Q})=H^*(M/S^1,\mathbb{Q})$.
\end{prop}
\begin{proof}
We will proceed by induction on the number of non-principal components.

To begin with, suppose $M$ does not have non-principal component. Since we assume there is no fixed point, then $S^1$ acts on $M$ freely and hence $H^*_{S^1}(M)=H^*(M/S^1)$.

Now suppose the proposition is true for any 3d fixed-point-free $S^1$-manifold with $k\geq 0$ non-principal components, and suppose $M$ has $k+1$ non-principal components. Let $C$ be a non-principal component together with an equivariant tubular neighbourhood $N$, then the complement $M'=M\setminus N$ has $k$ non-principal components and $H^*_{S^1}(M')=H^*(M'/S^1)$ according to our assumption. Let's also denote $L=M'\cap N$

The equivariant Mayer-Vietoris sequence for the union $M=M'\cup N$ and the ordinary Mayer-Vietoris sequence for the union $M/S^1=M'/S^1\cup N/S^1$ gives:
\begin{center}
\begin{tikzpicture}[description/.style={fill=white,inner sep=2pt}]
\matrix (m) [matrix of math nodes, row sep=2em,
column sep=1.5em, text height=1.5ex, text depth=0.25ex]
{H^{*-1}_{S^1}(M')\oplus H^{*-1}_{S^1}(N)  & H^{*-1}_{S^1}(L) & H^*_{S^1}(M) & H^{*}_{S^1}(M')\oplus H^{*}_{S^1}(N)  & H^{*}_{S^1}(L)\\
H^{*-1}(M'/S^1)\oplus H^{*-1}(N/S^1)  & H^{*-1}(L/S^1) & H^*(M/S^1) & H^{*}(M'/S^1)\oplus H^{*}(N/S^1)  & H^{*}(L/S^1)\\};
	\path[->]
	(m-1-1) edge (m-1-2)
	(m-1-2) edge (m-1-3)
	(m-1-3) edge (m-1-4)
	(m-1-4) edge (m-1-5)
	(m-2-1) edge (m-2-2)
	(m-2-2) edge (m-2-3)
	(m-2-3) edge (m-2-4)
	(m-2-4) edge (m-2-5)
	(m-1-1) edge (m-2-1)
	(m-1-2) edge (m-2-2)
	(m-1-3) edge (m-2-3)
	(m-1-4) edge (m-2-4)
	(m-1-5) edge (m-2-5);
\end{tikzpicture}
\end{center}
where the second and the fifth vertical maps are isomorphisms, because the intersection $L=M'\cap N$ does not touch non-principal orbits and consists of only principal orbits.

According to the Five Lemma in homological algebra, in order to prove that the middle vertical map is an isomorphism, we now need to prove the first and the fourth maps are isomorphisms. But we already have the isomorphism $H^*_{S^1}(M')=H^*(M'/S^1)$. So we only need to prove $H^*_{S^1}(N)=H^*(N/S^1)$.

In the 3d fixed-point-free $S^1$-manifold $M$, according to our detailed discussion in Subsection \ref{subsec:neighbour}, there are three cases for a non-fixed, non-principal component $C$, its equivariant neighbourhood $N$ and orbit space $N/S^1$. Note that, for each case, there is an equivariant deformation retraction $N \simeq C$, so we have $H^*_{S^1}(N)=H^*_{S^1}(C)$. Also recall that we have calculated $H^*_{S^1}(S^1/\mathbb{Z}_m,\mathbb{Q})=H^*(pt,\mathbb{Q})$.

\begin{center}
	\begin{tabular}{cccc}
		\toprule
		$C$ & $S^1/\mathbb{Z}_m$ & $S^1/\mathbb{Z}_2 \times S^1$ & $S^1/\mathbb{Z}_2 \times I$\\
		\midrule
		$N$ & $S^1 \times_{\mathbb{Z}_m} D^2$ & $\text{M\"{o}b}\times S^1$ & $\text{M\"{o}b}\times I$\\
		\midrule
		$N/S^1$ & $D^2/\mathbb{Z}_m$ & $I \times S^1$ & $I\times I$\\
		\midrule
		$H^*_{S^1}(N)=H^*_{S^1}(C)$ & $H^*(pt)$ & $H^*(S^1)$ & $H^*(I)$\\
		\midrule
		$H^*(N/S^1)$ & $H^*(D^2/\mathbb{Z}_m)$ & $H^*(S^1)$ & $H^*(I)$\\
		\bottomrule
	\end{tabular}
\end{center}

For second and the third case, it is clear that $H^*_{S^1}(N)=H^*(N/S^1)$. For the first case, the orbit space $D^2/\mathbb{Z}_m$, viewed as an ice-cream cone, has a deformation retract to the cone's tip $pt$, so $H^*_{S^1}(N)=H^*(pt)=H^*(D^2/\mathbb{Z}_m)=H^*(N/S^1)$.
\end{proof}

If a 3d $S^1$-manifold $M$ has fixed points, then every connected component of these fixed points is either a circle $S^1$ or an interval $I$. The calculation of $S^1$ equivariant cohomology of a general 3d $S^1$-manifold $M$ will be carried out by doing induction on the number of connected components of these fixed points. The beginning case of no fixed points is just the previous Proposition \ref{equivCohom0}. 

Suppose now that an $S^1$-manifold $M$ has $k>0$ connected components of fixed points. Let's choose any such a connected component $F$, with its equivariant neighbourhood $N$. If $F=S^1$, then $N=D\times S^1$; if $F=I$, then $N=D\times I$. In both cases, $N=D\times F$. If we set the complement $M'=M \setminus N$, then $M$ is attached equivariantly by $M'$ and $N=D\times F$ along $S^1 \times F$. The Mayer-Vietoris sequence of equivariant cohomology groups then gives
\[
\rightarrow H^*_{S^1}(M,\mathbb{Q}) \rightarrow H^*_{S^1}(M',\mathbb{Q}) \oplus H^*_{S^1}(D\times F,\mathbb{Q}) \rightarrow H^*_{S^1}(S^1\times F,\mathbb{Q})\rightarrow H^{*+1}_{S^1}(M,\mathbb{Q})\rightarrow
\]
However, since the $S^1$-action on $D\times F$ and $S^1\times F$ concentrates on their first components respectively, we have:
\begin{center}
\begin{tikzpicture}[description/.style={fill=white,inner sep=2pt}]
 \matrix (m) [matrix of math nodes, row sep=3em,
 column sep=2em, text height=1.5ex, text depth=0.25ex]
 {H^*_{S^1}(D\times F) & H^*_{S^1}(S^1\times F) & &\\
 H^*_{S^1}(D)\otimes H^*(F) & H^*_{S^1}(S^1)\otimes H^*(F)& &\\
 \mathbb{Q}[u]\otimes H^*(F) & H^*(F): & f(u)\otimes \alpha  & f(0)\cdot \alpha \\};
 \path[->]
 (m-1-1) edge (m-1-2)
 (m-2-1) edge (m-2-2)
 (m-3-1) edge (m-3-2);
 \path[-]
 (m-1-1) edge[double,double distance=2pt] (m-2-1)
 (m-1-2) edge[double,double distance=2pt] (m-2-2)
 (m-2-1) edge[double,double distance=2pt] (m-3-1)
 (m-2-2) edge[double,double distance=2pt] (m-3-2);
 \path[|->] (m-3-3) edge (m-3-4);
\end{tikzpicture}
\end{center}
where the upper 2 vertical isomorphisms are because of the cohomology of product spaces, the lower left vertical isomorphism is because of homotopy between $D$ and $pt$, and the lower right vertical isomorphism is because that the $S^1$ is a principal orbit.

The bottom map is obviously surjective, so is the top map $H^*_{S^1}(D\times F) \rightarrow H^*_{S^1}(S^1\times F)$. This means that the long exact sequence actually stops at $H^*_{S^1}(M') \oplus H^*_{S^1}(D\times F) \rightarrow H^*_{S^1}(S^1\times F)\rightarrow 0$. We then conclude that the long exact sequence reduces into the following short exact sequence:
\[
0\rightarrow H^*_{S^1}(M) \rightarrow H^*_{S^1}(M') \oplus 
\Big( \mathbb{Q}[u]\otimes H^*(F) \Big) \rightarrow H^*(F)\rightarrow 0
\]
where we have replaced the $H^*_{S^1}(D\times F)$ and $H^*_{S^1}(S^1\times F)$ by $\mathbb{Q}[u]\otimes H^*(F)$ and $H^*(F)$ respectively.

We can now consider all the $k$ components of fixed points $F_1,\,F_2,\,\ldots,\,F_k$, together with their equivariant tubular neighbourhood $N_1,\,N_2,\,\ldots,\,N_k$. If we set the complement $M_\circ = M \setminus \cup_i N_i$, an $S^1$-manifold without fixed points, then there is a short exact sequence of cohomology groups:
\begin{equation*}\label{eq:ShortSequence0}\tag{$\dagger$}
0\rightarrow H^*_{S^1}(M) \rightarrow H^*_{S^1}(M_\circ) \oplus \oplus_i 
\Big( \mathbb{Q}[u]\otimes H^*(F_i) \Big) \rightarrow  \oplus_i 
H^*(F_i) \rightarrow 0
\end{equation*}

Since $ M_\circ $ is fixed-point-free, $H^*_{S^1}(M_\circ,\mathbb{Q})=H^*(M_\circ/{S^1},\mathbb{Q})$ by Proposition \ref{equivCohom0}. To understand the orbit space $M_\circ/{S^1}$, we can compare it with the orbit space $M/{S^1}$.

\begin{lem}\label{homotopicOrbit}
Following the above notation, the two orbit spaces 
$M_\circ/{S^1}$ and $M/{S^1}$ are topologically homotopic. Especially, $H^*(M_\circ/{S^1},\mathbb{Q}) \cong H^*(M/{S^1},\mathbb{Q})$.
\end{lem} 
\begin{proof}
Since the majority of $M_\circ/{S^1}$ and $M/{S^1}$ is isomorphic, we only need to check what happens in an equivariant neighbourhood $N$ near an $S^1$-fixed component $F$ of $M$. 

Let $N'$ be an equivariant neighbourhood slightly larger than $N$. If we choose local $S^1$-equivariant coordinates properly, we can write $N' = D_1\times F$ and $N = D_{\frac{1}{2}}\times F$, where $D_1$ and $D_{\frac{1}{2}}$ are 2-dimensional disks of radii $1$ and $\frac{1}{2}$, such that $S^1$ acts on the disks by standard rotation.

Now $N'\setminus N = (D_1\setminus D_{\frac{1}{2}})\times F$ and $N' = D_1\times F$ are equivariant neighbourhoods of $M_\circ=M \setminus N$ and $M$ respectively. Their orbit spaces by the $S^1$-action give neighbourhoods $(N'\setminus N)/{S^1}$ and $N'/{S^1}$ of $M_\circ/{S^1}$ and $M/{S^1}$ respectively.

However,
\[
(N'\setminus N)/{S^1} = \Big((D_1\setminus D_{\frac{1}{2}})/{S^1}\Big)\times F = [\frac{1}{2},1)\times F
\]
and
\[
N'/{S^1} = \Big(D_1/{S^1}\Big)\times F = [0,1)\times F
\]
are homotopic. Thus $M_\circ /S^1$ and $M/S^1$ are homotopic.
\end{proof}

Finally, we can combine all the above discussions and get: 

\begin{thm}\label{equivCohom1}
Let $M$ be a compact connected 3d effective $S^1$-manifold(possibly with boundary), and $F$ be its fixed-point set(possibly empty), then there is a short exact sequence of cohomology groups in $\mathbb{Q}$ coefficients:

\begin{equation*}\label{eq:ShortSequence1}\tag{$\ddagger$}
0\rightarrow H^*_{S^1}(M) \rightarrow H^*(M/{S^1}) \oplus 
\Big( \mathbb{Q}[u]\otimes H^*(F) \Big) \rightarrow   
H^*(F) \rightarrow 0 
\end{equation*}

\end{thm}

\begin{proof}
If the fixed-point set $F$ is not empty, then we can use the short exact sequence \ref{eq:ShortSequence0}, and the replacement $H^*_{S^1}(M_\circ)=H^*(M_{\circ}/{S^1})=H^*(M/{S^1})$ because of the Lemma \ref{homotopicOrbit}. If the fixed-point set $F=\varnothing$ is empty, then $H^*(F)=0$. We just use the Proposition \ref{equivCohom0} which says $H^*_{S^1}(M) = H^*(M/{S^1})$. 
\end{proof}

\subsection{The ring and module structure}
By the short exact sequence \ref{eq:ShortSequence1} of Theorem \ref{equivCohom1}, we have the inclusion of cohomology groups: $H^*_{S^1}(M) \hookrightarrow H^*(M/{S^1}) \oplus 
\big( \mathbb{Q}[u]\otimes H^*(F) \big)$. But this inclusion is the direct sum of two restriction maps of cohomology rings, hence preserves ring structure. Therefore, we can describe the ring structure of $H^*_{S^1}(M)$ explicitly in terms of generators and relations from $H^*(M/{S^1})$ and $\mathbb{Q}[u]\otimes H^*(F)$.

For simplicity, we will focus on closed 3d $S^1$-manifolds. If $M$ does not have fixed points, then the Proposition \ref{equivCohom0} says that its equivariant cohomology ring is the cohomology ring of the orbit space. 

Thus we will only be interested in the case where $M$ has non-empty set of fixed points. According to the classification theorem, we can write $M = \big\{b=0;(\epsilon,g,f,s);(m_1,\,n_1),\ldots,(m_r,\,n_r)\big\}$ with $f>0$. Topologically, $M/{S^1}$ is a 2d surface of genus $g$, with $f+s>0$ boundary circles. 

Let's first give a description of the involved cohomologies $H^*(M/{S^1})$ and $\mathbb{Q}[u]\otimes H^*(F_i)$. 

The orbit space $M/{S^1}$ as a topological 2d surface of genus $g$, has $f$ boundary circles $\cup_{i=1}^f F_i$ from fixed components and $s$ boundary circles $\cup_{j=1}^{s} SE_j$ from the orbit spaces of special exceptional components. For a fixed circle $F_i=S^1,1\leq i \leq f$, we write $H^*(F_i,\mathbb{Q})=\mathbb{Q}\delta_i\oplus \mathbb{Q}\theta_i$, where $\delta_i$ and $\theta_i$ are generators of $H^0(F_i,\mathbb{Z})$ and $H^1(F_i,\mathbb{Z})$ respectively. Similarly, for $SE_j=S^1,1\leq j \leq s$, we write $H^*(SE_j,\mathbb{Q})=\mathbb{Q}\delta_{f+j}\oplus \mathbb{Q}\theta_{f+j}$. If the orbit space $M/S^1$ is orientable, i.e. $\epsilon = o$, though $\pm \theta_i$ are both generators for $H^1(F_i,\mathbb{Z})$, we only choose $\theta_i$ compatible with the boundary orientation on $F_i$. The same rule of choice also applies to $\theta_{f+j}$. Moreover, we can write $\mathbb{Q}[u]\otimes H^*(F_i) = \mathbb{Q}[u]\delta_i\oplus \mathbb{Q}[u]\theta_i$ such that every element of $\mathbb{Q}[u]\otimes H^*(F_i)$ can be expressed as $p_i(u)\delta_i+q_i(u)\theta_i$ for polynomials $p_i(u),p_i(u)\in \mathbb{Q}[u]$. 

Using the classic calculation of cohomology of 2d surfaces with boundaries, the cohomology $H^*(M/{S^1})$ has two different descriptions according to whether $M/{S^1}$ is orientable or not.

If $M/{S^1}$ is an orientable surface of genus $g$ with $f+s>0$ boundary circles, then it is homotopic to a wedge of $2g+f+s-1$ circles. Let's denote $\alpha_k,\beta_k,1\leq k \leq g$ for the generators of $H^1(-)$ of the $2g$ circles used in the polygon presentation of the surface $M/{S^1}$. Then we can write $H^*(M/{S^1})$ as a sub-ring of $\mathbb{Q}\delta_0\oplus \oplus_{k=1}^g\big(\mathbb{Q}\alpha_k\oplus\mathbb{Q}\beta_k\big) \oplus \big(\oplus_{i=1}^{f}\mathbb{Q}\theta_i\big)\oplus \big(\oplus_{j=1}^{s}\mathbb{Q}\theta_{f+j}\big)$, such that every element of $H^*(M/{S^1})$ can be expressed as $D\delta_0+\sum_k(A_k\alpha_k+B_k\beta_k)+\sum_i C_i\theta_i+\sum_j C_{f+j}\theta_{f+j}$ for $D,\,A_k,\,B_k,\,C_i,\,C_{f+j}\in \mathbb{Q}$, under the constraint that $\sum_k(A_k+B_k)+\sum_i C_i+\sum_j C_{f+j}=0$.

Moreover, we have the restriction maps to each fixed circle $F_i$:
\[
\mathbb{Q}[u]\otimes H^*(F_i)\rightarrow H^*(F_i): \quad p_i(u)\delta_i+q_i(u)\theta_i \mapsto  p_i(0)\delta_i+q_i(0)\theta_i
\]
and
\[
H^*(M/{S^1})\rightarrow H^*(F_i): \quad D\delta_0+\sum_{k=1}^g (A_k\alpha_k+B_k\beta_k)+\sum_{i=1}^f C_i\theta_i+\sum_{j=1}^s C_{f+j}\theta_{f+j} \mapsto D\delta_i+C_i\theta_i
\]

If $M/{S^1}$ is a non-orientable surface of genus $g$ with $f+s>0$ boundary circles, then it is homotopic to a wedge of $g+f+s-1$ circles. We can denote $\alpha_k,1\leq k \leq g$ for the generators of $H^1(-)$ of the $g$ circles used in the polygon presentation of the surface $M/{S^1}$. The description of the cohomology $H^*(M/{S^1})$ together with the restriction maps is similar to the orientable case, with the only difference that there is no $\beta_k,B_k$ for the non-orientable case.

Following the above notations, we get
\begin{thm}\label{equivCohom2}
	For a closed 3d $S^1$-manifold $M = \big\{b=0;(\epsilon = o,g,f,s);(m_1,\,n_1),\ldots,(m_r,\,n_r)\big\}$ with $f>0$ and an orientable orbit space $M/S^1$, an element of its equivariant cohomology $H^*_{S^1}(M)$ can be written as 
	\[
	\Big(D\delta_0+\sum_{k=1}^g (A_k\alpha_k+B_k\beta_k)+\sum_{i=1}^f C_i\theta_i+\sum_{j=1}^s C_{f+j}\theta_{f+j},\,\sum_{i=1}^f(p_i(u)\delta_i+q_i(u)\theta_i)\Big) \label{element} \tag{$*$}
	\]
	
	in $H^*(M/{S^1}) \oplus \oplus_i 
	\Big( \mathbb{Q}[u]\otimes H^*(F_i) \Big)$, under the relations
	\begin{enumerate}
		\item $\sum_{k=1}^g(A_k+B_k)+\sum_{i=1}^f C_i+\sum_{j=1}^s C_{f+j}=0$
		\item $p_1(0)=p_2(0)=\cdots=p_f(0)=D$
		\item $q_i(0)=C_i$ for each $i$
	\end{enumerate}
	Breaking the equivariant cohomology $H^*_{S^1}(M)$ into different degrees, we have
	\begin{itemize}
		\item $H^0_{S^1}(M)=\mathbb{Q}$
		\item $H^1_{S^1}(M)$ is a subgroup of $H^1(M/{S^1})\oplus\oplus_i H^1(F_i)$ consisting of elements
		\[
		\Big(\sum_{k=1}^g (A_k\alpha_k+B_k\beta_k)+\sum_{i=1}^f C_i\theta_i+\sum_{j=1}^s C_{f+j}\theta_{f+j}, \sum_{i=1}^f C_i\theta_i\Big)
		\]
		under the relation $\sum_{k=1}^g(A_k+B_k)+\sum_{i=1}^f C_i+\sum_{j=1}^s C_{f+j}=0$.
		\item $H^{\geq 2}_{S^1}(M)=\oplus_i\Big(\mathbb{Q}[u]_+\otimes H^*(F_i)\Big)$ where $\mathbb{Q}[u]_+$ consists of polynomials without constant terms.
	\end{itemize}
\end{thm}

\begin{proof}
	The expression $\eqref{element}$ of elements of $H^*_{S^1}(M)$ comes from the description of cohomologies $H^*(M/{S^1})$ and $\mathbb{Q}[u]\otimes H^*(F_i)$. The relations (1)(2)(3) are due to the theorem \ref{equivCohom1} that $H^*_{S^1}(M)$ is the kernel of the restriction map $H^*(M/{S^1}) \oplus \oplus_i 
	\Big( \mathbb{Q}[u]\otimes H^*(F_i) \Big) \rightarrow  \oplus_i 
	H^*(F_i)$. Thus, the images of restrictions are the same: $p_1(0)=p_2(0)=\cdots=p_f(0)=D$, and $q_i(0)=C_i$. Since the relations (1)(2)(3) only live in degree less than 2, we get the description of $H^*_{S^1}(M)$ in different degrees.
\end{proof}

\begin{rmk}\label{rmk:equivCohom3.5}
	For a closed 3d $S^1$-manifold $M = \big\{b=0;(\epsilon = n,g,f,s);(m_1,\,n_1),\ldots,(m_r,\,n_r)\big\}$ with $f>0$ and a non-orientable orbit space $M/S^1$. The
	explicit expression of elements of $H^*_{S^1}(M)$ is almost the same as the oriented case, with the only modification that there is no $\beta_k,B_k$ term.
\end{rmk}

\begin{thm}\label{equivCohom3}
	For a closed 3d $S^1$-manifold $M = \big\{b=0;(\epsilon,g,f,s);(m_1,\,n_1),\ldots,(m_r,\,n_r)\big\}$ with $f>0$, the graded ring structure of $H^*_{S^1}(M)$ is as follows:
	\begin{enumerate}
		\item $H^0_{S^1}(M)\otimes H^*_{S^1}(M) \overset{\cup}{\longrightarrow} H^*_{S^1}(M)$ and $H^*_{S^1}(M) \otimes H^0_{S^1}(M)  \overset{\cup}{\longrightarrow} H^*_{S^1}(M)$ are just scalar multiplication.
		\item $H^1_{S^1}(M)\otimes H^1_{S^1}(M) \overset{\cup}{\longrightarrow} H^2_{S^1}(M)$ is a zero map
		\item $H^1_{S^1}(M)\otimes H^{\geq 2}_{S^1}(M) \overset{\cup}{\longrightarrow} H^{\geq 3}_{S^1}(M)$ fits into a commutative diagram:
		\begin{center}
			\begin{tikzpicture}[description/.style={fill=white,inner sep=2pt}]
			\matrix (m) [matrix of math nodes, row sep=3em,
			column sep=2em, text height=1.5ex, text depth=0.25ex]
			{H^1_{S^1}(M)\otimes H^{\geq 2}_{S^1}(M) & H^{\geq 3}_{S^1}(M)\\
				\Big(\oplus_i H^1(F_i)\Big)\otimes \Big(\oplus_i\big(\mathbb{Q}[u]_+\otimes H^*(F_i)\big)\Big) & \oplus_i\Big(\mathbb{Q}[u]_+\otimes H^*(F_i)\Big)\\};
			\path[->]
			(m-1-1) edge (m-1-2)
			(m-2-1) edge (m-2-2)
			(m-1-1) edge (m-2-1);
			\path[-]
			(m-1-2) edge[double,double distance=2pt] (m-2-2);
			\end{tikzpicture}
		\end{center}
		where the left map is the restriction map ${H^1_{S^1}(M)\rightarrow \oplus_i H^1_{S^1}(F_i) = \oplus_i H^1(F_i)}$ tensored with the identification $H^{\geq 2}_{S^1}(M)=\oplus_i\Big(\mathbb{Q}[u]_+\otimes H^*(F_i)\Big)$, and the bottom map is the component-wise multiplication in $\oplus_i\Big(\mathbb{Q}[u]\otimes H^*(F_i)\Big)$. 
		\item $H^{\geq 2}_{S^1}(M)\otimes H^{\geq 2}_{S^1}(M) \overset{\cup}{\longrightarrow} H^{\geq 2}_{S^1}(M)$ is just the component-wise multiplication of $\oplus_i\Big(\mathbb{Q}[u]_+\otimes H^*(F_i)\Big)$
	\end{enumerate} 
\end{thm}
\begin{proof}We will explain the above breakdown one by one for the case when $M/S^1$ is orientable.
	\begin{enumerate}
		\item This is clear.
		\item From the Theorem \ref{equivCohom2}, $H^1_{S^1}(M)$ is generated by the basis $\alpha_j,\,\beta_j,\,\theta_i$, which have zero cup product among them.
		\item Similar to the above remark, the $H^1(M/{S^1})$ component of $H^1_{S^1}(M)\subset H^1(M/{S^1})\oplus\oplus_i H^1(F_i)$ has zero cup-product. So only the cup product involving $\oplus_i H^1(F_i)$ will survive.
		\item Since $H^{\geq 2}_{S^1}(M)=\oplus_i\Big(\mathbb{Q}[u]_+\otimes H^*(F_i)\Big)$, the cup product among $H^{\geq 2}_{S^1}(M)$ is inherited from $\oplus_i\Big(\mathbb{Q}[u]_+\otimes H^*(F_i)\Big)$.
	\end{enumerate}
The argument is exactly the same for the case when $M/S^1$ is non-orientable, because of the Remark \ref{rmk:equivCohom3.5}.
\end{proof}

Using the cup product of Theorem \ref{equivCohom3}, we can now describe the $H^*_{S^1}(pt)$-module structure of $H^*_{S^1}(M)$. 
\begin{thm}\label{equivCohom4}
	Following the notations of Theorem \ref{equivCohom3}, for a closed 3d $S^1$-manifold $M$ with non-empty set of fixed points , the forgetful map ${\pi:M\rightarrow pt}$ induces the map $\pi^*:H^*_{S^1}(pt)=\mathbb{Q}[u]\rightarrow H^*_{S^1}(M)$, with the image of the generator $u$ being $\pi^*(u)=\sum_i u\delta_i$. The generator $u$ acts on $H^*_{S^1}(M)$ by multiplying with $\pi^*(u)=\sum_i u\delta_i$ using the cup product of $H^*_{S^1}(M)$.
\end{thm}
\begin{proof}
	$u\in \mathbb{Q}[u]$ is of degree 2, so is $\pi^*(u)\in H^{\geq 2}_{S^1}(M)=\oplus_i\Big(\mathbb{Q}[u]_+\otimes H^*(F_i)\Big)$. Hence we only need to know the restriction of $\pi^*(u)$ from $H^*_{S^1}(M)$ to $H^*_{S^1}(F_i)$ for each fixed circle $F_i$.
	The commutative diagram of forgetful maps
	\begin{center}
		\begin{tikzpicture}[description/.style={fill=white,inner sep=2pt}]
		\matrix (m) [matrix of math nodes, row sep=2em,
		column sep=2em, text height=1.5ex, text depth=0.25ex]
		{F_i & M\\
			& pt\\};
		\path[right hook->] (m-1-1) edge (m-1-2);
		\path[->]
		(m-1-1) edge node[below] {$ \pi_i $}(m-2-2)
		(m-1-2) edge node[auto] {$ \pi $}(m-2-2);
		\end{tikzpicture}
	\end{center}
	induces the commutative diagram of maps between equivariant cohomologies
	\begin{center}
		\begin{tikzpicture}[description/.style={fill=white,inner sep=2pt}]
		\matrix (m) [matrix of math nodes, row sep=3em,
		column sep=2em, text height=1.5ex, text depth=0.25ex]
		{H^*_{S^1}(F_i) & H^*_{S^1}(M)\\
			& H^*_{S^1}(pt)\\};
		\path[<-]
		(m-1-1) edge (m-1-2)
		(m-1-1) edge node[below] {$ \pi_i^* $}(m-2-2)
		(m-1-2) edge node[auto] {$ \pi^* $}(m-2-2);
		\end{tikzpicture}
	\end{center}
	Thus the restriction of $\pi^*(u)$ from $H^*_{S^1}(M)$ to $H^*_{S^1}(F_i)$ is the image $\pi^*_i(u)$ via the map $\pi^*_i:H^*_{S^1}(pt)=\mathbb{Q}[u]\rightarrow H^*_{S^1}(F_i)=\mathbb{Q}[u]\delta_i\oplus \mathbb{Q}[u]\theta_i$. Since $F_i$ is a fixed component of the $S^1$-action on $M$, $u\in \mathbb{Q}[u]$ acts trivially on $H^*_{S^1}(F_i)$ with $\pi^*_i(u)=u\delta_i$.
	
	In conclusion, if we combine the contribution from all the fixed components $F_i$, we get $\pi^*(u)=\sum_i u\delta_i$.
\end{proof}

If a closed 3d $S^1$-manifold $M$ does not have fixed point, then the image $\pi^*(u)$ is in $H^2_{S^1}(M)=H^2(M/S^1)$ by the Proposition \ref{equivCohom0}. In this case, a condition for $\pi^*(u)=0$ is to make sure that $H^2(M/S^1)=0$.

\begin{prop}
For a closed 3d fixed-point-free $S^1$-manifold $M = \big\{b;(\epsilon,g,f=0,s);(m_1,\,n_1),\ldots,(m_r,\,n_r)\big\}$, if $\epsilon=n$ or $s>0$, then $H^2_{S^1}(M)=H^2(M/S^1)=0$, hence $\pi^*(u)=0$.
\end{prop}
\begin{proof}
By the classic calculation of cohomology of surfaces. A sufficient condition for $H^2(M/S^1)=0$ is that $M/S^1$ is non-orientable or has non-empty boundary, which corresponds to the condition: $\epsilon=n$ or $s>0$.
\end{proof}

If $\epsilon=o$ and $s=0$, then this is exactly the case of oriented Seifert manifold. The image $\pi^*(u)\in H^2_{S^1}(M)=H^2(M/S^1)$ is calculated by Niederkr\"{u}ger in his thesis (cf. \cite{Ni05} Theorem III.13).

\begin{thm}[Niederkr\"{u}ger, \cite{Ni05}]
Given an oriented Seifert manifold $M = \big\{b;(\epsilon=o,g,f=0,s=0);(m_1,\,n_1),\ldots,(m_r,\,n_r)\big\}$, let $l_i$ be the unique solution of $l_i n_i\equiv 1 \mod{m_i},\, 0<l_i<m_i$ for each coprime pair $(m_i,\,n_i)$. Then
\[
\pi^*(u)=b+\sum_{i=1}^r \frac{l_i}{m_i}\in H^2(M/S^1)=\mathbb{Q}
\]
\end{thm} 

\begin{rmk}
The rational number $b+\sum_{i=1}^r l_i/m_i$ is exactly the orbifold Euler characteristic of the oriented Seifert manifold, with integer $b$ contributed by the principal orbits and fraction $\sum_{i=1}^r l_i/m_i$ contributed by the exceptional orbits.
\end{rmk}

\subsection{The vector-space structure}
Since we are working in $\mathbb{Q}$-coefficient, the group structure of the equivariant cohomology $H^*_{S^1}(M)$ is simply the $\mathbb{Q}$-vector-space structure. In the short exact sequence \ref{eq:ShortSequence1}, we note that the surjective map $\mathbb{Q}[u]\otimes H^*(F) \longrightarrow H^*(F)$ by sending a polynomial $f(u)\in \mathbb{Q}[u]$ to its constant term $f(0)$, has a kernel $\mathbb{Q}[u]_+\otimes H^*(F)$, where $\mathbb{Q}[u]_+$ consists of polynomials without constant terms.

\begin{prop}\label{equivCohom5}
	Let $M$ be a compact connected 3d effective $S^1$-manifold(possibly with boundary), and $F$ be its fixed-point set(possibly empty), we get
	\[
	H^*_{S^1}(M) \cong H^*(M/{S^1})\oplus \Big(\mathbb{Q}[u]_+\otimes H^*(F)\Big) \quad \mbox{as vector spaces}
	\]
	where $\mathbb{Q}[u]_+$ consists of polynomials without constant terms.
\end{prop}

\begin{proof}
	From Theorem \ref{equivCohom1}, we see that $H^*_{S^1}(M)$ is the kernel of the map $H^*(M/{S^1}) \oplus 
	\Big( \mathbb{Q}[u]\otimes H^*(F) \Big) \rightarrow H^*(F)$. But the map $\mathbb{Q}[u]\otimes H^*(F) \rightarrow H^*(F)$ has the kernel $\mathbb{Q}[u]_+\otimes H^*(F)$. So we can write $H^*_{S^1}(M)$ as a direct sum of $H^*(M/{S^1})$ and $\mathbb{Q}[u]_+\otimes H^*(F)$.
\end{proof}

\begin{rmk}
The above expression of $H^*_{S^1}(M)$ as a direct sum usually does not preserve the ring structure, unless $F=\varnothing$, i.e. $M$ is fixed-point-free.
\end{rmk}

\begin{rmk}
	If the fixed-point set $M^{S^1}=F=\cup_i F_i$ is non-empty, then the orbit space $M/S^1$ has boundaries, so $H^{*\geq 2}(M/S^1)=0$. Also note $\mathbb{Q}[u]_+\otimes H^*(F_i)$ has degrees at least 2. So the above theorem says that when $M^{S^1}\not = \varnothing$, we have
	\begin{itemize}
		\item[(1)]
		$H^{*\leq 1}_{S^1}(M) = H^*(M/{S^1})$ is determined by the orbit space and $H^{*\geq 2}_{S^1}(M)= \oplus_i \Big(\mathbb{Q}[u]_+\otimes H^*(F_i)\Big)$ is determined by the fixed-point set. 
		\item[(2)]
		Since $H^*(S^1)$ contributes to both even and odd degrees, but $H^*(I)$ only contributes to even degrees. We have $\#\{F_i=S^1\}=\text{dim}\,H^3_{S^1}(M)$ and $\#\{F_i=I\}=\text{dim}\,H^2_{S^1}(M)-\text{dim}\,H^3_{S^1}(M)$.
	\end{itemize}
\end{rmk}

\subsection{Equivariant Betti numbers and Poincar\'e series}
Given an $S^1$-manifold $M$, we can calculate its equivariant Betti numbers $b^k_{S^1}=\mathrm{dim}\, H^k_{S^1}(M)$ and the equivariant Poincar\'e series $P^M_{S^1}(x) = \sum_{k=0}^\infty b^k_{S^1} x^k$ .

When a closed 3d $S^1$-manifold $M$ has neither fixed points nor special exceptional orbits, i.e. $f=s=0$, also called Seifert manifold, its orbit space $M/{S^1}$ is a closed 2d orbifold of genus $g$. By Proposition \ref{equivCohom0}, $H^*_{S^1}(M,\mathbb{Q})=H^*(M/{S^1},\mathbb{Q})$ and the classic calculation of cohomology of closed surfaces, we have

\begin{prop}\label{thm:bettiPoinc0}
For a closed 3d $S^1$-manifold $M$ without fixed points nor special exceptional orbits, i.e. $M = \big\{b;(\epsilon,g,f=0,s=0);(m_1,\,n_1),\ldots,(m_r,\,n_r)\big\}$, the equivariant Poincar\'e series are $1+2gx+x^2$ if $M$ is orientable, or $1+gx$ if $M$ is non-orientable.
\end{prop}

When the set of fixed points or special exceptional orbits is non-empty, we will get:

\begin{thm}\label{bettiPoinc}
For a closed 3d $S^1$-manifold $M = \big\{b;(\epsilon,g,f,s);(m_1,\,n_1),\ldots,(m_r,\,n_r)\big\}$ with $f+s>0$(hence $b=0$), its equivariant Betti numbers are
\begin{align*}
b^0_{S^1} &= 1\\
b^1_{S^1} &= 
\begin{cases}
\hfill 2g+f+s-1 \hfill &\text{if $\epsilon=o$}\\
\hfill g+f+s-1 \hfill &\text{if $\epsilon=n$} \\
\end{cases} \\
b^{2k}_{S^1} &= f \quad \mbox{for $k\geq 1$}\\
b^{2k+1}_{S^1} &= f \quad \mbox{for $k\geq 1$}
\end{align*}
with the equivariant Poincar\'e series
\[
P^M_{S^1}(x) = \sum_{k=0}^\infty b^k_{S^1} x^k = 
\begin{cases}
\hfill 1+(2g+f+s-1)x+f\cdot\frac{x^2+x^3}{1-x^2} \hfill &\text{if $\epsilon=o$}\\
\hfill 1+(g+f+s-1)x+f\cdot\frac{x^2+x^3}{1-x^2}\hfill &\text{if $\epsilon=n$}\\
\end{cases}
\]
\end{thm}

\begin{proof}
By Theorem \ref{equivCohom5}, the equivariant cohomology of $M$ is
\[
H^*_{S^1}(M) \cong H^*(M/{S^1}) \oplus \oplus_{i=1}^f\Big(\mathbb{Q}[u]_+\otimes H^*(F_i)\Big) \quad \mbox{as vector spaces}
\]
where $\mathbb{Q}[u]_+$ is the set of polynomials without constant terms and $F=\cup_i^f F_i$ is the union of fixed circles.

Note that, $M/S^1$ is a 2d surface of genus $g$ with $f+s>0$ boundaries. Its Poincar\'e series are ${1+(2g+f+s-1)x}$ if $\epsilon=o$, or  ${1+(g+f+s-1)x}$ if $\epsilon=n$, using the classic result on the cohomology of 2d surface with boundary. For each $\mathbb{Q}[u]_+\otimes H^*(F_i),\, 1\leq i \leq f$, it's easy to see that the Poincar\'e series are $\frac{x^2}{1-x^2}\cdot (1+x)$.

Then we can calculate the equivariant Poincar\'e series $P^M_{S^1}(x)$ and equivariant Betti numbers $b_{S^1}^*$ of $M$ additively from those of $M/S^1$ and $F_i$.
\end{proof}

\subsection{Equivariant formality}
Using the explicit description of the ring and module structures, we can determine when a closed 3d $S^1$-manifold is equivariant formal in the following sense.

\begin{dfn}
A $G$-action on a manifold $M$ is \textbf{equivariantly formal}, if the equivariant cohomology $H_G^*(M)$ is a free $H_G^*(pt)$-module.
\end{dfn}

When talking about equivariant formality, we will only be interested in the case of closed manifolds in this paper.

\begin{thm}\label{equivFormal}
A closed 3d $S^1$-manifold $M = \big\{b;(\epsilon,g,f,s);(m_1,\,n_1),\ldots,(m_r,\,n_r)\big\}$ is $S^1$-equivariantly formal if and only if $f>0,\,b=0$ and
\[
\begin{cases}
\hfill g=s=0 \text{ or } g=0,\,s=1 \hfill& \text{if }\epsilon=o\\
\hfill g=1,\,s=0 \hfill& \text{if }\epsilon=n\\
\end{cases}
\]
\end{thm}
\begin{proof}
For the necessity, when $M$ is $S^1$-equivariantly formal, $H_{S^1}^*(M)$ is a free $H_{S^1}^*(pt)$-module. Since the polynomial ring $H_{S^1}^*(pt)=\mathbb{Q}[u]$ is infinite dimensional, so is $H_{S^1}^*(M)$. Therefore it must have non-empty fixed-point set to generate elements of degree to the infinity, so $f>0$, and hence $b=0$. 

The polynomial ring $\mathbb{Q}[u]$, with $u$ of degree 2, has non-decreasing Betti numbers in odd degrees and even degrees respectively. Hence, so does any free $H_{S^1}^*(pt)=\mathbb{Q}[u]$-module. 
\begin{eqnarray*}
b^{2k}_{S^1} &\leq& b^{2k+2}_{S^1} \quad \mbox{for $k\geq 0$}\\
b^{2k+1}_{S^1} &\leq& b^{2k+3}_{S^1} \quad \mbox{for $k\geq 0$}
\end{eqnarray*}
Especially, we will verify $b^1_{S^1} \leq   b^3_{S^1}$ by substituting our calculation of the Betti numbers $b_{S^1}^*$ from Theorem \ref{bettiPoinc}. 

When $\epsilon = o$, we get $2g+f+s-1 \leq  f$, or equivalently, $2g+s \leq 1$. Here, $s$ as the number of special exceptional components in $M$, is non-negative; $g$ as the genus of an orientable surface, is also non-negative. These constraints force $g=s=0$ or $g=0,\,s=1$.

When $\epsilon = n$, we get $g+f+s-1 \leq  f$, or equivalently, $g+s \leq 1$. Here $s$ again is non-negative. But $g$ as the genus of a non-orientable surface, is strictly positive. These constraints force $g=1,\,s=0$.

For the sufficiency, let's first assume $f>0,\,b=0$. 

When $\epsilon = o,\,g=0,\,s=0$, there are no $\alpha_k,\beta_k,\theta_{f+j}$ terms, by Theorem \ref{equivCohom2}. Also note that $D\delta_0+\sum_{i=1}^f C_i \theta_i$ can be absorbed into $\sum_i(p_i(u)\delta_i+q_i(u)\theta_i)$ because of the relations(2)(3) in that Theorem. Hence there is a much nicer expression of an element of the equivariant cohomology $H_{S^1}^*(M)$:
\[
\sum_{i=0}^f\big(p_i(u)\delta_i+q_i(u)\theta_i\big) \in \mathbb{Q}[u]\otimes H^*(F)
\]
under the relations:
\[
p_1(0)=p_2(0)=\cdots=p_f(0) \mbox{ and } \sum_{i=0}^f q_i(0)=0 
\]
This is indeed a free $\mathbb{Q}[u]$-module, since we can find its $\mathbb{Q}[u]$-module generators without extra relations:
\begin{eqnarray*}
& \sum_{i=0}^f \delta_i  &\mbox{(1 term in deg 0)}\\
& \theta_1-\theta_2,\,\ldots,\,\theta_1-\theta_f  &\mbox{($f-1$ terms in deg 1)}\\
& u(\delta_1-\delta_2) ,\,\ldots,\, u(\delta_1-\delta_f) &\mbox{($f-1$ terms in deg 2)}\\
& u\sum_{i=0}^f \theta_i  &\mbox{(1 term in deg 3)}
\end{eqnarray*}

When $\epsilon = o,\,g=0,\,s=1$, there are no $\alpha_k,\beta_k$ terms and only one $\theta_{f+1}$ term among the $\theta_{f+j}$ terms, by Theorem \ref{equivCohom2}. 
Again we can absorb $D\delta_0+\sum_{i=1}^f C_i \theta_i$ into $\sum_i(p_i(u)\delta_i+q_i(u)\theta_i)$. Moreover, the condition (1) in Theorem \ref{equivCohom2} says $C_{f+1}+\sum_{i=0}^f q_i(0)=0$, so we can absorb $C_{f+1}\theta_{f+j}$ into $\sum_i q_i(u)\theta_i$. Hence, every element of the equivariant  cohomology $H_{S^1}^*(M)$ can be expressed as:
\[
\sum_{i=0}^f\big(p_i(u)\delta_i+q_i(u)\theta_i\big) \in \mathbb{Q}[u]\otimes H^*(F)
\]
under the relations:
\[
p_1(0)=p_2(0)=\cdots=p_f(0)
\]
This is indeed a free $\mathbb{Q}[u]$-module, since we can find its $\mathbb{Q}[u]$-module generators without extra relations:
\begin{eqnarray*}
	& \sum_{i=0}^f \delta_i  &\mbox{(1 term in deg 0)}\\
	& \theta_1,\,\ldots,\,\theta_f  &\mbox{($f$ terms in deg 1)}\\
	& u(\delta_1-\delta_2) ,\,\ldots,\, u(\delta_1-\delta_f) &\mbox{($f-1$ terms in deg 2)}\\
\end{eqnarray*}

When $\epsilon = n,\,g=1,\,s=0$, there is only one $\alpha_1$ term among the $\alpha_k$'s, but no $\beta_k,\theta_{f+j}$ terms, by Theorem \ref{equivCohom2} and the remark next to it. Again we can absorb $D\delta_0+\sum_{i=1}^f C_i \theta_i$ into $\sum_i(p_i(u)\delta_i+q_i(u)\theta_i)$. Moreover, the condition (1) in Theorem \ref{equivCohom2} says $A_1+\sum_{i=0}^f q_i(0)=0$, so we can absorb $A_1\alpha_1$ into $\sum_i q_i(u)\theta_i$. Hence, every element of the equivariant  cohomology $H_{S^1}^*(M)$ can be expressed as:
\[
\sum_{i=0}^f\big(p_i(u)\delta_i+q_i(u)\theta_i\big) \in \mathbb{Q}[u]\otimes H^*(F)
\]
under the relations:
\[
p_1(0)=p_2(0)=\cdots=p_f(0)
\]
This is indeed a free $\mathbb{Q}[u]$-module, since we can find its $\mathbb{Q}[u]$-module generators without extra relations:
\begin{eqnarray*}
& \sum_{i=0}^f \delta_i  &\mbox{(1 term in deg 0)}\\
& \theta_1,\,\ldots,\,\theta_f  &\mbox{($f$ terms in deg 1)}\\
& u(\delta_1-\delta_2) ,\,\ldots,\, u(\delta_1-\delta_f) &\mbox{($f-1$ terms in deg 2)}
\end{eqnarray*}
\end{proof}

If we focus on the oriented case with $\epsilon=o,s=0$, then
\begin{cor}
A closed oriented 3d $S^1$-manifold $M = \big\{b;(\epsilon=o,g,f,s=0);(m_1,\,n_1),\ldots,(m_r,\,n_r)\big\}$ is $S^1$-equivariantly formal if and only if $f>0,\,b=0,\,g=s=0$.
\end{cor}

When a closed 3d $S^1$-manifold $M$ satisfies $\{\epsilon=o, f>0, b=0,g=s=0\}$, we get its Poincar\'e series using Theorem \ref{bettiPoinc}:
\[
P^M_{S^1}(x) = 1+(f-1)x+f\cdot\frac{x^2+x^3}{1-x^2}
\]
On the other hand, the enumeration of $\mathbb{Q}[u]$-module generators in the above proof of Theorem \ref{equivFormal} gives the Poincar\'e series
\begin{eqnarray*}
P^M_{S^1}(x) &=& \big(1+(f-1)x+(f-1)x^2+x^3\big)\cdot P^{pt}_{S^1}(x)\\
             &=& \big(1+(f-1)x+(f-1)x^2+x^3\big)\cdot (1+x^2+x^4+\cdots)\\
             &=& \frac{1+(f-1)x+(f-1)x^2+x^3}{1-x^2}
\end{eqnarray*}
However, one can easily check that these two expressions are the same. 

Similarly, when a closed 3d $S^1$-manifold $M$ satisfies $\{\epsilon=o,f>0,b=0,g=0,s=1\}$ or $\{\epsilon=n,f>0,b=0,g=1,s=0\}$, we get its Poincar\'e series using Theorem \ref{bettiPoinc}:
\[
P^M_{S^1}(x) = 1+fx+f\cdot\frac{x^2+x^3}{1-x^2}
\]
On the other hand, the enumeration of $\mathbb{Q}[u]$-module generators in the above proof of Theorem \ref{equivFormal} gives the Poincar\'e series
\begin{eqnarray*}
P^M_{S^1}(x) &=& \big(1+fx+(f-1)x^2\big)\cdot P^{pt}_{S^1}(x)\\
             &=& \big(1+fx+(f-1)x^2\big)\cdot (1+x^2+x^4+\cdots)\\
             &=& \frac{1+fx+(f-1)x^2}{1-x^2}
\end{eqnarray*}
One can also easily check that these two expressions are the same. 

\section{Acknowledgment}
The author would like to thank Professor Victor Guillemin for suggesting this project and the guidance throughout, and thank Professor Jonathan Weitsman for continuous encouragement and inspiring discussions. 

\bibliographystyle{amsalpha}

\end{document}